\documentclass[reqno,12pt]{amsart}
\usepackage{amssymb,verbatim, epic,eepic,epsfig,amsbsy,amsmath,amscd,mathrsfs}

\theoremstyle{plain}

\newtheorem{Thm}{Theorem}
\newtheorem{thm}{Theorem}[section]
\newtheorem{prop}[thm]{Proposition}
\newtheorem{lemma}[thm]{Lemma}
\newtheorem{cor}[thm]{Corollary}

\theoremstyle{definition}

\newtheorem{rmk}[thm]{Remark}

\theoremstyle{remark}

\def\be{\begin{equation}}
\def\ee{\end{equation}}
\def\bes{\begin{equation*}}
\def\ees{\end{equation*}}
\def\ba{\begin{align}}
\def\ea{\end{align}}
\def\bas{\begin{align*}}
\def\eas{\end{align*}}

\newcommand{\Z}{\mathbb{Z}}

\newcommand{\R}{\mathbb{R}}
\newcommand{\BZ}{\mathbb{Z}}
\newcommand{\BQ}{\mathbb{Q}}
\newcommand{\Q}{\mathbb{Q}}
\newcommand{\Zq}{\BZ[q^{\pm1}]}
\newcommand{\Zx}{\BZ[\xi]}

\newcommand{\Zqf}{\BZ[q^{\pm\frac14}]}

\newcommand{\Zqz}{\BZ[z^{\pm1}, q^{\pm1}]}

\def\Qx{\BQ[\xi]}
\def\cH{\mathcal H}

\newcommand{\bk}{\mathbf{k}}
\renewcommand{\a}{\mathbf{a}}

\def\BZq{\BZ[q^{\pm 1}]}
\def\ord {\mathrm{ord}}

\newcommand{\cU}{{\mathcal U}}

\def\bk{{\bf k}}
\def\b{{Q_0}}

\def\bn{{\bf n}}

\newcommand{\qbinom}[2]{\genfrac{[}{]}{0pt}{}{#1}{#2}}
\newcommand{\binomq}[2]{\binom{#1}{#2}_{\!\!q}}
\newcommand{\binomx}[2]{\binom{#1}{#2}_{\!\!\xi}}
\newcommand{\cb}[1]{\{#1\}}
\newcommand{\tF}[1]{\tilde F^{(#1)}}

\newcommand{\sn}{\operatorname{sn}}
\newcommand{\ex}{\operatorname{ev}_\xi}
\newcommand{\ev}{\operatorname{ev}}
\newcommand{\diag}{\operatorname{diag}}
\newcommand{\sumxi}{{\sum_{n_1,\dots, n_m}}}
\newcommand{\sumx}{{\sum_{n}}}
\renewcommand{\l}[1]{\mathcal{L}_{#1}}

\newcommand{\tr}{\operatorname{tr}}

\newcommand{\id}{\operatorname{id}}

\renewcommand{\L}{\operatorname {\mathbf L}}

\def\la{\langle}
\def\ra{\rangle}

\def\Pe{P^{(1)}}
\def\cD{\mathcal D}

\def\lam{\lambda}
\def\ve{\varepsilon}

\def\F{\mathcal F}
\def\U{\mathcal U}

\def\cD{\mathcal D}
\def\Zzq{\BZ[z^{\pm 1}, q^{\pm 1}]}
\def\Zxf{\BZ[ \xi^{ 1/4}]}

\newcommand{\Uq}[1]{\cU_q^{(#1)}}

\begin{document}

\title[On the integrality of the WRT invariants]{On the integrality of
the Witten--Reshetikhin--Turaev\\[2mm]
 3--manifold invariants}
\author{Anna Beliakova}
\address{I-Math, University of Zurich\newline
\indent Winterthurerstrasse 190\newline
 \indent 8057 Zurich,
Switzerland}
\email{anna@math.uzh.ch}
\author{Qi Chen}
\address{Department of Mathematics\newline
\indent Winston-Salem State University\newline
\indent Winston Salem, NC 27110
USA}
\email{chenqi@wssu.edu}
\author{Thang  Le}
\address{Department of Mathematics\newline
         \indent Georgia Institute of Technology\newline
         \indent Atlanta, GA 30332-0160, USA}
\email{letu@math.gatech.edu}

\thanks{A.B. and T.L. are  supported  by the Swiss and American
National Science Foundations.
\newline
\indent{\em Mathematics Classification (2010):} Primary 57M27; Secondary 57M25.
\newline
}



\begin{abstract} We prove that the $SU(2)$ and $SO(3)$
 Witten--Reshetikhin--Turaev invariants
of any 3--manifold with any colored link inside at any root of unity
 are algebraic integers.
\end{abstract}

\maketitle


\addtocounter{section}{-1}


\section{Introduction}
\label{sec.intro}
In the late 80s, Witten \cite{Witten} and Reshetikhin-Turaev \cite{RT} associated
with any closed  oriented 3-manifold $M$ (possibly with a colored link inside),
 any root of unity $\xi$ and
any compact  Lie group $G$  a complex number $\tau_M^G(\xi)$,
called
the quantum or WRT invariant of  $M$. 

For more than 20 years, the
problem of integrality of the WRT invariants
has been intensively studied.
The interest to this problem was drawn by the
theory of perturbative  3-manifold invariants
generalizing those of Casson and Walker \cite{Ohtsukibook},
by the construction of Integral Topological Quantum Field Theories  \cite{g,gm} and their
 topological applications
 and more recently,
by attempts to  categorify the WRT invariants \cite{Kho}.

In the case $G=SU(2)$, there is a projective version $\tau_M^{SO(3)}(\xi)$,
introduced by Kirby and Melvin \cite{KM} and
defined at  roots of unity  of
 {\em odd} order. This projective version, 
when defined, 
determines the $SU(2)$ version.

In this paper we completely solve the integrality problem for both $SO(3)$ and $SU(2)$
versions of the WRT  invariant for all 3-manifolds with arbitrary links inside.
Before  stating our results,
let us give a brief introduction into the history of this subject.

In 1995 Murakami \cite{Mu}  established  the integrality of the WRT $SO(3)$-invariant
for rational homology 3-spheres at roots of unity of {\em prime orders}.
This result was extended to all 3-manifolds by Masbaum and Roberts \cite{MR}.
Masbaum and Wenzl \cite{MW}, and independently Takata and Yokota \cite{TY},   proved the
  integrality of the projective WRT $SU(n)$-invariant
for all 3-manifolds,
 always under  the assumption that
 the orders of the roots of unity are {\em prime}. Finally the third author \cite{Le2} established the integrality
of the projective WRT invariant
 associated with any  compact simple Lie group, again at roots of unity of prime orders.

The case for the roots of unity of {\it non-prime} orders is more complicated.
The first integrality result for {\em all roots of unity}
was obtained by
Habiro \cite{ha} in  the case of $SU(2)$ and  {\em integral homology 3-spheres}.
Habiro's proof relies on
the existence of the unified
invariant for   integral homology 3-spheres 
as an element of   Habiro's ring,
 a certain cyclotomic
completion of the polynomial ring $\Z[q]$.
This unified invariant is
a kind of generating function for  the set of WRT $SU(2)$ invariants
at all roots of unity. The integrality  
 in this approach follows directly from the general properties of 
Habiro's ring.

Habiro and the third author \cite{HL}  subsequently defined
the unified WRT invariant for all {\em simple Lie groups}
 and {\em integral homology 3-spheres}, thus proving
that the WRT invariant of any {\em integral homology 3-sphere} associated to any simple Lie group
and any root of unity is always an  algebraic integer. 
However, the case of manifolds other than homology spheres was unknown, even with $G=SU(2)$.

In this paper we give
 a complete solution for the integrality problem
 for {\em all 3-manifolds} with {\em arbitrary} link inside
at {\it all roots} of unity
for the case of the group $SU(2)$. Our invariants are normalized 
as in \cite{KM}
and we  show
that integrality in that case implies integrality for all other
normalizations used in the literature.

\begin{Thm}\label{main1}
The  WRT $SU(2)$-invariant of any
 3-manifold $M$ with any colored link inside
at any root of unity 
  is an algebraic integer.
\end{Thm}

\begin{Thm}\label{main2}
The  WRT $SO(3)$-invariant of any
 3-manifold $M$ with  any colored link inside
at any root of unity of odd order is an algebraic integer.
\end{Thm}



Theorem \ref{main2} is a  generalization of a result in
  \cite{bl} to  manifolds which contain a link inside. 
However, we give here a new  independent proof
along the same lines as in the
 $SU(2)$ case. 
Theorem \ref{main1} is the main result of the paper. 
The key new ideas used in the proofs are the following.


One of the main tools is a significant generalization of 
some divisibility result (Theorem \ref{thm1}) which was originally
obtained in \cite{Le4} using a
number-theoretical  identity of Andrews' \cite{Andrews}, whose special 
cases are the classical
Rogers-Ramanujan identities.

Further, to include the case of even colored links in  3-manifolds,
we had to introduce a new basis for the 
Grothendieck ring of the quantum $sl(2)$, which is orthogonal to
the odd part of the center with respect to the Rosso form. This led to an important
new  result (Theorem 1.1) generalizing that of Habiro,
which states that the colored Jones polynomial 
can be presented as a sum of integral ``blocks''. This result
is proved in the Appendix, and it is of independent interest in the  quantum link
invariant theory.

For manifolds obtained by surgery along links with diagonal linking matrix
we show that the contribution of each integral block to the WRT invariant is  integral
by using our main tool (Theorem \ref{thm1}). 
The general case can be reduced to the diagonal one by
using some  classification results for linking pairings.
However, it is more demanding in the $SU(2)$ case  than in the $SO(3)$
one, since the linking pairings on
abelian groups of even order are  more complicated \cite{KK}.

As a byproduct, we generalize the relationship between
$SU(2)$ and $SO(3)$ invariants at {\rm odd} roots of unity
to the case when a 3-manifold  contains an arbitrary colored link inside.
For empty links and links colored by the fundamental representation,
this relationship was established in \cite{KM} and \cite{MR}, respectively.

At the moment of this writing, our proof cannot be generalized
to higher-ranked Lie groups because we do not have an analog of
Theorem 1.1  (splitting into integral blocks) in those cases.
The paper is  as self-contained as possible. The only 
two  results used without proofs here are
\cite[Theorem 7]{Le4}
and \cite[Theorem 2]{bbl2}.

We organize this paper as follows.
In Section \ref{notations} we fix notations, recall the definition
of the WRT invariant and state a generalization of Habiro's result.
The main strategy of our proofs is explained in
Section \ref{strategy}.
In Section \ref{div} we prove  some divisibility results 
for generic values of the quantum parameter.
Formulas related to roots of unity are proved in Section \ref{root}.
Section 4 deals with the symmetry principle and the splitting of the
$SU(2)$ invariant at odd roots of unity into the product
of the $SO(3)$ and Deloup's invariants.
Section \ref{secd} discusses how 
to construct 3-manifolds that can be obtained by surgery along links with
diagonal linking matrices.
The last two sections are devoted to the proofs of Theorem \ref{main2} and Theorem \ref{main1}, respectively.

\section{The colored Jones polynomial and the WRT invariant} \label{notations}

\subsection{Notations}
Let $q^{1/4}$ be a formal parameter. Set
$$
\cb n := q^{n/2}-q^{-n/2}, \quad \cb n! := \prod_{i=1}^n \cb i, \quad [n] := \frac{\cb n}{\cb1}, \quad
\qbinom nk := \frac{\cb n!}{\cb k! \cb {n-k}!}\ ,
$$
and
$$
(z;q)_m = \prod_{i=0}^{m-1} (1-q^iz),\quad
\binomq mn := \frac{(q^{m-n+1}; q)_n}{(q; q)_n} = q^{(m-n)n/2}\qbinom mn \ .
$$
Throughout this paper, let $\xi$ be a primitive root of unity of order $r$ 
and $\xi^{1/4}$ be a complex number such that $(\xi^{1/4})^4 =\xi$. There are 4 possible choices for $\xi^{1/4}$, and we
will make some restrictions later.

When working in the $SO(3)$ case, we will always assume that $r\geq 3$ is {\it odd}.
In the $SU(2)$ case, $r\geq 2$ will be
 an arbitrary positive integer.

For $f\in \BQ[q^{\pm 1/4}]$, we define
the following evaluation map
$$
\ex(f):=f|_{q^{1/4}=\xi^{1/4}}.
$$
It should be noted that although we write $\ev_\xi(f)$, this quantity depends on the
choice of a 4-th root $\xi^{1/4}$ of $\xi$.

If $f$ is a function on positive integers $n_1,\dots,n_k$ with values in $\BQ[q^{\pm 1/4}] $, we define
$$
{\sum_{n_1,\dots,n_k}}^{\xi, SO(3)} \, f := \frac1{4^k}\sum_{\substack{n_j = 0 \\ n_j \text{ odd}}}^{4r-1}\ex(f)\ ,
\quad\quad
{\sum_{n_1,\dots,n_k}}^{\xi, SU(2)}\,  f := \frac{1}{4^k}\sum_{n_j = 0}^{4r-1} \ex(f) \ .
$$


All 3-manifolds in this paper are supposed to be closed and oriented.
Every link in a 3-manifold is framed, oriented and has ordered components.


\subsection{The colored Jones polynomial}\label{cj} Suppose
$L$ is a framed oriented link in $S^3$ with $m$ ordered components. For an $m$-tuple of positive integers
$\bn=(n_1,\dots,n_m)$, one has the colored Jones polynomial $J_L(\bn)\in \Zqf$, see e.g. \cite{Tu,MM}. 
The number $n_i$ is usually called the color of the $i$-th component, and stands for
the $n_i$-dimensional irreducible $sl_2$-representation in the theory of quantum link invariants.
 We use the normalization so that
$J_U(n) = [n]$ where $U$ is the unknot with $0$ framing. It is well known that
if $\tilde L$ is obtained from $L$ by increasing the framing on the $i$-th component by 1 then
\be
J_{\tilde L}(\bn) = q^{\frac{n^2_i-1}4} J_L(\bn)\ .
\label{e002}
\ee

Although there are fractional powers $q^{\pm 1/4}$, there exists an integer $a = a(L,\bn)$ such that $J_{L}(\bn)\in q^{a/4} \, 
\BZ[q^{\pm 1}]$.
For a precise formula of $a$ see \cite{Le_Duke}. This formula implies that
if all the colors $n_j$'s are odd, then $J_L(\bn) \in \BZ[q^{\pm1}]$.

\subsection{Habiro's expansion and its generalization}\label{se13}
\def\bs{\mathbf s}
Assume that $L\sqcup L'$ is a framed link in $S^3$ with disjoint sublinks $L$ and $L'$. 
Suppose $L$ has $m$ ordered components and $L'$
has $l$ ordered components.
Fix an $l$-tuple of  positive integers $\bs=(s_1, \dots, s_l)$, and
let's consider $J_{L\sqcup L'}(\bn,\bs)$
 as a function on $m$-tuples $\bn=(n_1,\dots,n_m)$. 
Since $\bs$ is  fixed, we will remove it from the notation for
simplicity. 
The function $J_{L\sqcup L'}(\bn)$ can be rearranged into
another  function $c_{L\sqcup L'}(\bk)$
  generalizing an important result of Habiro
\cite[Theorem 8.2]{ha}.

To state the result we need to introduce a few notations.
Let $\tilde \ell_{ij}$ be the linking
number between the $i$-th component of $L$ and the $j$-th component of $L'$.
For any $i=1,\dots,m$, we define
\begin{equation} \ve_i \in \{0,1\}  \quad \text{ by } \quad
 \ve_i := \sum^l_{j=1} {\tilde \ell}_{ij} (s_j-1) \pmod 2 \ .\quad
\label{3306}
\end{equation}

\begin{thm}\label{prop.habiro} Assume that $L\sqcup L' \subset S^3$ is  as described above.
Suppose that $L$ has 0 linking matrix.
 Then for every $m$-tuple $\bk=(k_1,\dots,k_m)$ of non-negative integers with $k=\max(k_1,\ldots,k_m)$
there exists
\be
c_{L\sqcup L'}(\bk)\; \in \;
\frac{(q^{k+1}; q)_{k+1}}{1-q}\; \Z[q^{\pm1/4}]\
\label{e007}
\ee
 such that for every $m$-tuple $\bn=(n_1,\dots,n_m)$ of non-negative integers, 
\begin{equation}\label{eq.jj}
J_{L\sqcup L'}(\bn)
= \sum_{k_i\ge 0} c_{L\sqcup L'}(\bk) \;\prod_{i=1}^m \qbinom{n_i+k_i}{2k_i+1}
\;\cb{k_i}!\;\frac{\lam^{\ve_i}_{n_i}}{\lam^{\ve_i}_{k_i+1}}
\end{equation}
where $\lam_n=q^{n/2}+q^{-n/2}$.

\end{thm}

For the  case when all $\ve_i=0$, or, in particular, when all $s_i$'s are odd,
the statement is equivalent to
 \cite[Theorem 3]{bbl1}. A proof of Theorem \ref{prop.habiro} is given in Appendix A.
Note that for a fixed $\bn$ the right hand side of \eqref{eq.jj} is a finite sum because $\qbinom{n+k}{2k+1}=0$ if $n \leq k$.

This is the presentation of the colored Jones polynomial as a sum of integral blocks mentioned in Introduction.
The existence of $c_{L\sqcup L'}(\bk) \in \BQ(q^{1/4})$ that satisfies
\eqref{eq.jj} is easy to prove. The real content of Theorem \ref{prop.habiro} is the integrality \eqref{e007}.

\subsection{The WRT invariant}\label{wrt}
We review here the definition of the
 WRT $SU(2)$ invariant of a 3-manifold $M$ with a colored link $L'$ inside
  \cite{RT} and its  $SO(3)$ version \cite{KM}.

We use the convention that the pair
$(M, L')$ is obtained  from $(S^3, L')$ by surgery along
 $L$. Here $L'$ is an $\bs$-colored framed link.
 For $G=SU(2)$ or $G=SO(3)$
set
\be\label{sta}
F^G_{L\sqcup L'}(\xi) : = \sumxi^{\xi,G} \left\{J_{L\sqcup L'}(\bn)\prod_{i=1}^m[n_i]\right\}\ .
\ee
For simplicity,
 we assume here that all entries of $\bs$ are odd if  $G=SO(3)$. 
In  general for $G=SO(3)$, we have to multiply
\eqref{sta} by a power of $\xi$, depending on the 
linking matrix of $L'$ and the parity of colors. This is done
  in Section 4.2. Since the additional factor is a unit, it
does not affect integrality.

We want to emphasize that although it is not explicit from the notation, 
\eqref{sta} depends on a choice of a 4-th root $\xi^{1/4}$ of $\xi$.

It is known that  $F^G_{L\sqcup L'}(\xi)$ is invariant under the handle slide move 
 and if normalized appropriately, is an invariant of the pair $(M,L')$.

Let $U^\pm$ be the unknot with $\pm1$ framing. It is easy to see that $F^G_{U^-}(\xi)$ is the complex conjugate of $F^G_{U^+}(\xi)$.
Let
$$ \cD^G:= |F^G_{U^+}(\xi)| = \sqrt {F^G_{U^+}(\xi) \, F^G_{U^-}(\xi)}\ \ .$$
This number is called the rank of a TQFT in \cite{Tu}.
We normalize by dividing \eqref{sta} by certain powers of 
$F^G_{U^\pm}(\xi)\neq 0$. Hence, we want to know when $F^G_{U^\pm}(\xi)\neq 0$.
The following is probably known. For completeness we include a proof in Section \ref{sec3.3}.

\begin{lemma} \label{lem12} One has $F^G_{U^\pm}(\xi) =0$ if and only if 
\be \text{$G=SU(2)$ and $\xi^{1/4}$ has order $2 \, \ord(\xi)=2r$.} \tag{$\star$}
\label{nd}
\ee
\end{lemma}

In \cite{KM, RT} and \cite{Lickorish} it is assumed that $\ord(\xi^{1/4}) = 4\, \ord(\xi)$. However, there are other cases when $F^G_{U^\pm}(\xi) \neq 0$.
Here we  consider all of them. 

In the entire paper we will assume that condition (\ref{nd}) is not satisfied, so that $F^G_{U^\pm}(\xi)\neq 0$.

Then the WRT  invariant of the pair $(M,L')$ is defined by
\be\label{inv}
\tau^{G}_{M, L'}(\xi) = \frac{F^G_{L\sqcup L'}(\xi)}{(F^G_{U^+}(\xi))^{\beta_+}(F^G_{U^-}(\xi))^{\beta_-}\,  (\cD^G)^\beta},
\ee
where  $\beta_+,\beta_-$ and $\beta$ are respectively   the number of positive, negative, and 0
eigenvalues of the linking matrix of $L$.

The invariant $\tau^{G}_{M, L'}(\xi)$ is multiplicative with respect to
connected sum. If $-M$ is $M$ with the reverse orientation, then $\tau^G_{-M}(\xi)$ is the complex conjugate of $\tau^G_{M}(\xi)$, and
$\tau^G_{S^3}(\xi)=1$.

\begin{rmk}
We will prove later that $\cD^G \in \Z[\xi^{1/4}, e_8]$, where $e_8=\exp(\pi\sqrt{-1}/4)$.
Note that $\Z[\xi^{ 1/4}, e_8]=\Z[\exp(2\pi \sqrt{-1}/t)]$, where $t=8r$ if $r$ is odd
and $t=4r$ if $r$ is even. In the last case, $e_8\in \Z[\xi^{ 1/4}]$. 


Hence a priori,
$\tau^{G}_{M, L'}(\xi) \in \BQ(\xi^{1/4}, e_8)$. 
Since the ring of integers of $\BQ(\xi^{1/4}, e_8)$ is $\BZ[\xi^{1/4}, e_8]$, 
our invariant is algebraically integral if it belongs to
 $ \BZ[\xi^{1/4}, e_8]$.

Further, if $G=SO(3)$, $M$ is a rational homology 3-sphere, and all the $s_i$'s are odd, 
then $\tau^{SO(3)}_{M, L'}(\xi) \in \BQ(\xi)$ by definition. 
So, in that case integrality means that 
$\tau^{SO(3)}_{M, L'}(\xi) \in \Z[\xi]$ for any root of unity $\xi$ of odd order.
\end{rmk}

\subsection*{Relations with other invariants}
If we put $\xi^{1/4}:= \exp(\pi \sqrt {-1}/2r)$, then our invariant $\tau_M^{SU(2)}(\xi)$ and $\tau_M^{SO(3)}(\xi)$ are 
respectively $\tau_r(M)$ and $\tau'_r(M)$ in \cite{KM}. In that case, our $\cD^{SU(2)}$ equals to $b^{-1}$ in the notation of \cite{KM}.

Again, if $\xi^{1/4}= \exp(\pi \sqrt {-1}/2r)$,
the original Reshetikhin-Turaev invariant \cite{RT} differs from $\tau_r(M)$ by a multiplication with a certain
root of unity, so this does not affect  integrality.

The set of invariants considered in \cite{MR} coincides with ours assuming  $r$ is an odd prime.
More precisely, the invariants $I_{2r}(M)$  and $I_{r}(M)$, defined in \cite{MR} 
as functions of a variable $A$,
coincide with ours $\tau^{SU(2)}_M(\xi)$
and $\tau^{SO(3)}_M(\xi)$ after setting
$A=-\xi^{1/4}$ and $A=-\xi^{(r+1)^2/4}$, respectively.
 At these roots of unity, the $SO(3)$ invariants determines those for $SU(2)$.

In \cite{Lickorish}, Lickorish 
chose a different normalization and worked with 
$\tau^{SU(2)}_M(\xi)(\cD^{SU(2)})^\beta$ in our notation.
Clearly, integrality of this  invariant will follow from the integrality of
$\tau^G_{M,L'}(\xi)$.

\subsection{Diagonal case}
Of particular importance  is the  case when the
 linking matrix of $L$ is  a diagonal matrix $\diag(b_1, \ldots, b_m)$, $b_i\in \Z$ for any $i$.
Let $L^0$ be the framed link obtained from $L$ by switching
all the framings to $0$. Recall from \eqref{3306}
that for $1\leq i\leq m$, $ \ve_i := \sum^l_{k=1} {\tilde \ell}_{ik} (s_k-1) \pmod 2$.
Using \eqref{e002} and \eqref{eq.jj},  we
can rewrite $F^G_{L\sqcup L'}(\xi)$ as follows:
\begin{equation}\label{fl}
F^G_{L\sqcup L'}(\xi) = \sum_{k_i\ge 0} \ex (c_{L^0\sqcup L'}(\bk)/\cb 1^m )
\prod_{i=1}^m H^G(k_i,b_i, \ve_i)
\end{equation}
where 
\begin{equation}\label{eq.h}
H^G(k, b, \ve) := \sumx^{\xi, G} q^{b\frac{n^2-1}4} \qbinom{n+k}{2k+1}\, \frac{\lam^\ve_n}{\lam^\ve_{k+1}}
\cb{k}!\cb{n}\ .
\end{equation}

By \eqref{sta} and \eqref{eq.h} we also have
\be\label{tri}
F^G_{U^\pm}(\xi) = \frac{H^G(0,\pm1,0)}{\ex(\cb 1)} \ .
\ee
From \eqref{fl} and \eqref{tri} we get the following.

\begin{prop}\label{ptau-diag}  Suppose the linking matrix of $L$ is  a diagonal matrix $\diag(b_1, \ldots, b_m)$
with exactly  $t$ non-zero  elements $b_1, \ldots, b_t$. 
Assume the entries of $\bs$ are 
 odd when $G=SO(3)$.
Then
\be
\tau^G_{M, L'}(\xi) =  
\sum_{k_i = 0}^{\lfloor\frac{r-2}2\rfloor} \ex (c_{L^0\sqcup L'}(\bk)) \prod_{i=1}^t \frac{H^G(k_i, b_i,\ve_i)}{H^G(0, \sn(b_i),0)}
\prod^m_{i=t+1} \frac{H^G(k_i, 0,\ve_i)}{\ex(\{1\}) \cD^G} \ ,
\label{e_diag}
\ee
where  $\sn(b_i)$ is the sign of $b_i$. 
\end{prop}

Note that in the above sum the index $k_i$ is from 0 to $\lfloor\frac{r-2}2\rfloor$. This is because
$(\xi^{k+1};\xi)_{k+1}=0$ when $k>(r-2)/2$, so $\ex (c_{L^0\sqcup L'}(\bk))=0$ when
$k=\max\{k_i\}>(r-2)/2$ according to
Proposition \ref{prop.habiro}.

To allow an arbitrary coloring $\bs$ of $L'$ for $G=SO(3)$, 
we have to multiply the right hand side
of \eqref{e_diag} by a unit, defined in Section 4.2.

We say that $M$ is
{\em diagonal of prime type},  when $M$ can be obtained
by surgery along a link with diagonal linking matrix whose entries are (up to a sign)
0, 1 or prime powers.

\subsection{Strategy for the proof of Theorems \ref{main1} and \ref{main2}}\label{strategy}
We first prove the integrality of $\tau^G_{M,L'}(\xi)$ for the case when
$M$ is diagonal of prime type.
By \eqref{e_diag},  in this case it suffices to show that
$$\frac{H^G(k, b,\ve)}{H^G(0, \sn(b),0)}\quad{\text
 {and}}\quad \frac{H^G(k,0,\ve)}{(1-\xi)\cD^G}$$
are  algebraic integers when $0\le k \le \lfloor\frac{r-2}2\rfloor$.
This is proved in Proposition \ref{lem}
for $G=SO(3)$ and in Proposition \ref{lem42} for $G=SU(2)$, under assumptions $r$ 
is odd and even, respectively. 


The general case can be reduced to the diagonal one 
of prime type by applying some
 standard results on diagonalization, presented in Section \ref{secd}.
 Roughly speaking, $M\#M$ becomes
diagonal of prime type after  adding a diagonalizing manifold $N$, which is a connected sum of
some simple lens spaces. In the $SO(3)$ case, this already solves the problem, since
 the WRT invariant of $N$ is invertible. In the $SU(2)$ case, the WRT invariants of $N$ might be 0. We show
that there is an odd colored link $L \subset N$ such that
 $\tau^{SU(2)}_{N,L}$ is integral and non-zero. However, another difficulty arises since $\tau^{SU(2)}_{N,L}$ is not invertible.
To overcome this difficulty  we will look at the connected sum of many copies of $M\#M$ with $(N,L)$, which we
will show to be diagonal of prime type.
Further, we make substantial use of the fact that in any Dedekind domain, every
ideal has a unique prime factorization.

The case $G=SU(2)$ and $r$ odd is solved in Section 4. There we generalize the splitting formula of
Kirby and Melvin \cite{KM} by
showing  that
the $SU(2)$ invariant of any 3-manifold with a colored link 
inside at a root of unity of odd order
is a product of the $SO(3)$ invariant and another integer invariant, previously defined by Deloup.

\section{Basic divisibility: the case of generic $q$} \label{div}

 In this section we establish a divisibility result for generic $q$ which will help us to prove
that each factor of \eqref{e_diag} is integral.

\subsection{The ideal $I_k$}\label{secI}
Let $I_k$ be the ideal of $\Zqz$ generated by $(q^az; q)_k$ for all $a \in \BZ$. This ideal plays an important role in the theory of quantum invariants, see \cite{ha,Le4,HL}.

We will use the following characterization of $I_k$, which is Proposition 4.3 of \cite{Le4}.

\begin{prop}\label{ideal}
The ideal $I_k$ is the set of all $f \in \Zqz$ such that $f(q^b,q)$ is divisible by $(q;q)_k$ for every $b \in \BZ$.
\end{prop}


 We will often use the following $q$-binomial formula
\begin{equation}\label{eq.newton}
(q^az; q)_k = \sum_{j=0}^k (-1)^j \binomq kj q^{\binom j2 + aj}\, z^j\ .
\end{equation}


\subsection{Divisibility for generic $q$}\label{sec.div}
For  a positive integer $k$ let
\begin{equation}\label{xn}
X_k := \frac{(q; q)_k}{(q; q)_{\lfloor k/2\rfloor}} = \prod_{j=\lfloor k/2\rfloor+1}^k (1-q^j)\ .
\end{equation}

A map $Q:\BZ \to \BZ$ is said to be an {\em integral quadratic form} if
$$Q(n) = a_2 n^2 + a_1 n + a_0$$ for some
integers $a_2, a_1, a_0$. 

For a quadratic form $Q$ let
 $\l Q : \Zqz\to \Zq$ be the $\Zq$-linear map defined by
$$
\l Q (z^j) = q^{Q(j)}\ .
$$
Note that this map is not an algebra homomorphism if $a_2$ or $a_0\ne 0$.

Let $\sigma$ be the $\Z[q^{\pm 1}]$-algebra automorphism of $\Zqz$ defined by 
$\sigma(z) = z^{-1}$. An important observation is that
if $a_1=0$, then $\l Q \sigma=\l Q$.

\begin{thm}\label{thm1}
For an arbitrary integral quadratic form $Q$  and any $f(z,q) \in I_k$, the element
$\l Q(f)$ is divisible by $X_k$, i.e.
$$\l Q(f)\;\in\; X_k \;\Z[q^{\pm 1}]\, .$$ 
\end{thm}

\begin{rmk}
This theorem will be used substantially. It is a generalization of
 \cite[Theorem 7]{Le4}, which was proved  with the help of
Andrews' identity.  The case $Q(n)=n^2$ of Theorem \ref{thm1} 
appeared in \cite{HL} for the construction of the unified WRT invariant.
\end{rmk}

\begin{proof}  By the definition of $I_k$, it is enough to consider the case when $f = z^m (q^a z;q)_k$.
Suppose
$$Q(n)=a_2 n^2 + a_1 n + a_0\ .$$
Let $\b(n) = a_2 n^2$. Rewriting $f$ as a sum
 with help of \eqref{eq.newton}, one can show
\begin{equation*}
\l Q (q^{-a_1m}z^m(q^{a-a_1}z; q)_k) = q^{a_0} \l\b (z^m(q^az; q)_k).
\end{equation*}
Hence, the substitution of $q^{-a_1}z$ for $z$ 
transforms $\l Q$ into $\l\b$. Without loss of generality,
 we can further  assume $Q=\b$.

The rest of the proof is by induction on $k$.

The case $k=1$ is trivial.
Suppose that the statement holds true
for $k-1$.

Since
$$
z^m(q^{a+1} z; q)_k - z^m(q^az; q)_k = q^a(1-q^k)z^{m+1}(q^{a+1}z; q)_{k-1}\ ,
$$
we see that, by the induction hypothesis, $\l\b(z^m(q^{a+1} z; q)_k)$ is divisible by $X_k$ if and only if
$\l\b(z^m(q^az; q)_k)$ is. Therefore we only need to show the statement for a single value of $a$.
We will take $a=-\lfloor k/2\rfloor$.
The cases of odd and even $k$ will be considered separately.

Suppose $k=2l+1$. Then $a=-l$ and
$\l\b(z^m(q^{-l} z; q)_{2l+1})$
is divisible by $X_k = (q^{l+1}; q)_{l+1}$ by Lemma \ref{lemma1} (b) below.

Now suppose $k=2l$. Then $a=-l$ and we need to show that for every integer $m$, $X_k$
divides $\l\b(B(m, l))$ where
$$
B(m,l) := z^m(q^{-l} z; q)_{2l} \ .
$$
Since
$$
B(m,l) - q^l B(m+1, l) = z^m (q^{-l}z; q)_{2l+1}
$$
and $X_k = (q^{l+1}; q)_l$ divides $\l\b(z^m (q^{-l}z; q)_{2l+1})$
by Lemma \ref{lemma1} (b) below, it is enough to show that $X_k$ divides $\l\b(B(m, l))$ for only a single value of $m$. We choose $m=-l$,
and we will  show that $X_k=(q^{l+1};q)_l$ divides $\l\b(B(-l, l))$.

Using that $\sigma$ is the algebra automorphism of $\Zqz$ sending $z$ to $z^{-1}$,
we get
\bas
(\id + \sigma) B(-l, l) & = z^{-l} (q^{-l} z; q)_{2l} + z^l (q^{-l} z^{-1}; q)_{2l} \\
& = z^{-l} (q^{-l} z; q)_{2l} + z^{-l} q^{-l} (q^{1-l} z; q)_{2l} \\
& = z^{-l} (q^{1-l} z; q)_{2l-1} (1-q^{-l} z + q^{-l} (1 - q^l z)) \\
& = z^{-l} (q^{1-l} z; q)_{2l-1}(1-z)(1+q^{-l})\\
& = -q^{-l} (1+q^l)\, y_{l-1} \ ,
\end{align*}
where
$$
y_l := z^{-l}(1-z^{-1})(q^{-l}z; q)_{2l+1} = (-1)^l q^{-\frac{l(l+1)}{2}} \prod^l_{j=0} (z-q^j)(z^{-1}-q^j)\ .
$$
From $\l\b\sigma = \l\b$ it follows that
$$
2\l\b(B(-l, l)) = \l\b((\id+\sigma)B(-l, l)) =- q^{-l} (1+q^{l}) \l\b(y_{l-1}) \ ,
$$
which is divisible by $2(1+q^l)(q^l; q)_l = 2(q^{l+1}; q)_l$ thanks to Lemma \ref{lemma1} (a).
This completes the induction, whence the proof.
\end{proof}

\begin{lemma}\label{lemma1}
With the same notations as above we have
\begin{enumerate}
	\item[(a)] if $f\in \Zqz$ is invariant under $\sigma$,
	then $2(q^{l+1}; q)_{l+1}$ divides $\l\b(f y_l)$;
	\item[(b)] for any $f\in \Zqz$, $\l\b\big ( (q^{-l}z; q)_{2l+1} \, f\big)$ is divisible by $(q^{l+1}; q)_{l+1}$.
\end{enumerate}
\end{lemma}

\begin{proof}
(a) First we prove the case $f=1$. We will show that this case 
follows from \cite[Theorem 7]{Le4}, which was proved by using the Andrews identity.
In fact we have
\begin{align}
y_l  & = z^{-l}(1-z^{-1})(q^{-l}z; q)_{2l+1} \notag \\
&= z^{-l} (q^{-l}z; q)_{2l+1} - z^{-l-1} (q^{-l}z; q)_{2l+1}\ . \label{e0021}
\end{align}
It is easy to see that the two terms of the right hand side of \eqref{e0021} are related by
\be\label{stand}
\sigma \big( z^{-l} (q^{-l}z; q)_{2l+1}  \big)  =    - z^{-l-1} (q^{-l}z; q)_{2l+1}\ .\ee
 Hence
\begin{align*}
\l\b(y_l)
& = 2 \l\b(z^{-l}(q^{-l}z; q)_{2l+1})  \nonumber \\
  & = 2
\sum_{j=0}^{2l+1} (-1)^j \qbinom {2l+1}j q^{a_2(j-l)^2}\ ,
\end{align*}
which, according to \cite[Theorem 7]{Le4}, is divisible by 
$$2\frac{\cb{2l+1}!}{\cb l!}= 2(-1)^{l+1} q^{-(l+1)(3l+2)/4}(q^{l+1}; q)_{l+1}$$
in $\BZ[q^{\pm1/2}]$.
Since $\l\b(y_l)$ and $2(q^{l+1}; q)_{l+1}$ are both in $\Zq$, (a) is true when $f=1$.

Consider the general case. Since $\sigma(f)=f$,
$f$ is a polynomial in $(z+z^{-1})$. It is enough to prove
(a) for $f=(z+z^{-1})^m$.
Since
$$
(z+z^{-1})y_l = y_{l+1} + (q^{l+1}+q^{-l-1})y_l\ ,
$$
$(z+z^{-1})^m y_l$ is a $\Zq$-linear combination of $y_l, y_{l+1},\ldots, y_{l+m}$.
Because $(q^{l+i};q)_{l+i}$ is divisible by $(q^{l+1};q)_{l+1}$ for every positive integer
$i$, the case
$f= (z+z^{-1})^m$ follows from the case $f=1$.

(b) For non-negative integer $m$, using
$$\sigma \big( z^{m} (q^{-l}z; q)_{2l+1}  \big)  =    - z^{-m-2l-1} (q^{-l}z; q)_{2l+1}\ ,$$
we get
$$
2\l\b(z^m(q^{-l}z; q)_{2l+1}) = \l\b\left((\id+\sigma)z^m(q^{-l}z; q)_{2l+1}\right)
= \l\b(y_l \sum_{j=-m-l}^{m+l} z^j)\ ,
$$
which is divisible by $2(q^{l+1}; q)_{l+1}$ according to (a).
Similar argument works for negative $m$.
\end{proof}

The following corollary is sometimes more convenient than Theorem \ref{thm1}.

\begin{cor}\label{cor}
For every positive integer $k$ and every integral quadratic form $Q$, $X_k$ divides
$$
\sum_{j=0}^k (-1)^j \binomq kj q^{Q(j)+\binom j2}\ .
$$
\end{cor}

\begin{proof}
From (\ref{eq.newton}) we have
$$ \l Q \big( (z; q)_k\big)  = \sum_{j=0}^k (-1)^j \binomq kj q^{Q(j)+\binom j2}\ .$$
By Theorem \ref{thm1}, the left hand side is divisible by $X_k$, and so is the right hand side.
\end{proof}

\subsection{Polynomials with $q$-integer values}
We also need a generalization of the following classical result
in the theory of polynomials with integer values:
If $f(z_1,\dots,z_n)\in \BQ[z_1,\dots, z_n]$
takes integer values whenever $z_1,\dots,z_n$ are integers,
then $f$ is a $\BZ$-linear combination of
$\prod_{i=1}^n \binom{z_i}{k_i}$,
$k_i \in \Z_{\ge0}$. 

Let us formulate a $q$-analog of this fact.

\begin{prop}\label{multi}
If $f(z_1,\dots, z_n) \in \BQ(q)[z_1,\dots, z_n]$ satisfies
$f(q^{m_1}, \dots, q^{m_n})\in \BZ[q^{\pm 1}]$
 for every $m_1,\dots,m_n \in \BZ$, then $f$ is a $\BZ[q^{\pm 1}]$-linear combination of
$$\prod_{i=1}^n \frac{(z_i;q)_{k_i} }{(q;q)_{k_i}} \quad {\text with}\quad k_i \in \Z_{\ge0} \, .$$

\end{prop}

\begin{proof}
  The elements $z_{\bk}:= \prod_{i=1}^n (z_i;q)_{k_i} /(q;q)_{k_i}$, with $\bk=(k_1,\dots, k_n)  \in \Z_{\ge0}^n$,
 form a $\BQ(q)$-basis of $\BQ(q)[x_1,\dots, x_n]$. Hence there are $c_\bk \in \BQ(q)$ such that
$f = \sum_{\bk\in \Z_{\ge0}^n} c_\bk \, z_\bk$.
Only a finite number of $c_\bk$'s
are non-zero. We will show that $c_\bk \in \BZ[q^{\pm1}]$ by induction on $|\bk|:= k_1 +\dots+ k_n$.

Suppose $\bk=0$. Let $z_1=z_2=\cdots
= z_n=1$, then $z_\bk=0$ unless $\bk=0$. Hence $c_0= f(1,1,\dots, 1) \in \BZ[q^{\pm 1}]$.

Suppose $c_\bk\in \BZ[q^{\pm 1}]$ for $|\bk| < l$. The $z_\bk$'s with $|\bk|< l$ will be called terms of lower orders.
Consider a $\bk =(k_1,\dots,k_n)$ with $|\bk|=l$. Note that  when
$z_i = q^{-k_i}$,
$z_{(a_1,\dots,a_n)} =0$ if for some $i$ one has $a_i > k_i$, and $z_{\bk}
= \pm 1$.
Hence

$$ f( q^{-k_1}, \dots , q^{-k_n}) = \pm c_\bk + \text{terms of lower orders}.$$
By induction, the terms of lower orders are in $\BZ[q^{\pm 1}]$. Since the left hand side is in $\BZ[q^{\pm 1}]$,
we conclude that $c_\bk \in \BZ[q^{\pm 1}]$.
\end{proof}

\begin{cor}\label{append} For any integer $a$, the element
$(q^a z_1 z_2;q)_k$
is a $\Z[q^{\pm 1}]$--linear combination of terms
$$\frac{(q;q)_k}{(q;q)_{k_1}(q;q)_{k_2}} (z_1;q)_{k_1}\,
(z_2;q)_{k_2}\; $$
with $k_1,k_2\leq k$.
\end{cor}

\begin{proof}
The evaluation of
$$\frac{(q^a z_1 z_2;q)_k}{(q;q)_k}
\in \Q(q)[z_1, z_2]$$
 at $z_i=q^{m_i}$ belongs to $\Z[q^{\pm 1}]$
for any $m_i$. Applying Proposition \ref{multi} we get the result.
Note that $k_1$ and $k_2$ should be less than or equal to $k$ by degree reason.
\end{proof}

%
%
%

\section{Basic results: the case of roots of unity}\label{root} 
In this section we prove
a basic divisibility result for the case when $q$ is a root of unity $\xi$ of order $r$.
In Subsection \ref{simp} we reduce the integrality of $H^G(k,b,\ve)/H^G(0,\pm1,0)$ to that of a simpler quotient.


For $x, y \in \BQ(\xi^{1/4}, e_8)$ we write
$ x \sim y $ if $x/y$ is an invertible element in $\BZ[\xi^{1/4}, e_8]$.

We use the notation $\bar k=r-1-k$,  and
$$
 \binomx mn := \ex \binomq mn, \quad   O_\xi := (\xi;\xi)_{\lfloor \frac{r-1}{2}\rfloor}, \quad x_k := \prod_{j=\lfloor k/2 \rfloor +1 }^k (1- \xi^j)=\ev_\xi(X_k),
$$
where $X_k$ is defined in \eqref{xn}.

\subsection{Divisibility}
 The main divisibility result at roots of unity is formulated below.
\begin{prop} \label{div1} For every integral quadratic form $Q$  and $f(z,q) \in I_k$ with $0\le k < r$ we have
$$\sum_{n=0}^{r-1} \xi ^{Q(n)}\, f(\xi^n,\xi) \in {x_k O_\xi}\cdot \BZ[\xi]\ .$$
\end{prop}

We need the following lemma  for the proof of Proposition \ref{div1}.

\begin{lemma}\label{prop1}
For any integers $a$, $k$, with $0\le k < r$, and integral quadratic form $Q$, the element
$$y=\sum_{n=0}^{r-1} \xi^{Q(n)} \binomx{n+a}k$$
 is divisible by $x_{\bar k}$.
\end{lemma}

\begin{proof} Using $1-\xi^m = -\xi^m (1-\xi^{r-m})$ we have
\begin{align}
\binomx {n+k}n = \frac{(\xi^{k+1}; \xi)_n}{(\xi; \xi)_n}
& = (-1)^n \xi^{kn+n(n+1)/2} \frac{(\xi^{r-k-n}; \xi)_n}{(\xi; \xi)_n} \notag \\
&= (-1)^n\xi^{nk+n+\binom n2} \binomx{\bar k}n. \label{e650}
\end{align}

Set $n' = n+a-k$. One has
\begin{align}
y  &= \sum_{n=0}^{r-1} \xi^{Q(n)} \binomx{n'+k}k  \notag \\
  & = \sum_{n'=0}^{r-1} (-1)^{n'} \xi^{Q'(n')+\binom {n'}2}
\binomx{\bar k}{n'} \label{e561}
\end{align}
by \eqref{e650}, where $Q'(n') = Q(n'-a+k)+n'k+n'$. In the right hand side of \eqref{e561}, the index $n'$ actually runs from 0 to $\bar k-1$, since $\binomx{\bar k}{n'}=0$ if $n' \ge \bar k$. The right hand side of \eqref{e561}  is divisible by $x_{\bar k}$  by Corollary \ref{cor}.
\end{proof}

\begin{proof}[Proof of Proposition \ref{div1}] Since the set $\{ z^d(z q^a;q)_k : d,a \in \BZ\}$ 
spans $I_k$ over $\BZq$, we can assume that $f=z^d(z q^a;q)_k$. Then
\bas \frac{\sum_{n=0}^{r-1} \xi ^{Q(n)}\, f(\xi^n,\xi)}{x_k O_\xi} &= \frac{(\xi;\xi)_{\lfloor k/2\rfloor}}{O_\xi} \sum_{n=0}^{r-1} \xi ^{Q(n)+ dn }\, \binomx{n+a-1}{k} \\
& \in \frac{(\xi;\xi)_{\lfloor k/2\rfloor}}{O_\xi}\,  x_{\bar k}\, \Zx\,  \quad \text{ (by Lemma \ref{prop1})} ,
\end{align*}
which is in  $\Zx$ by Lemma \ref{ab} (f) below.

\end{proof}

\subsection{The ring $\BZ[\xi]$} It is known that $\Zx$ is a Dedekind domain with field of fractions $\Qx$.

\begin{lemma}\label{ab}
a) If $(a,r)=(b,r)$ then $(1-\xi^a)\sim (1-\xi^b)$ in $\BZ[\xi]$.

b) One has $ (\xi;\xi)_{r-1} = r$.

c) Suppose $y \in \BQ[\xi]$ and  $y^s \in \BZ[\xi]$ for some  positive integer $s$. Then  $y \in \BZ[\xi]$.

d) Suppose $y,z \in \BQ[\xi]$ with $z \neq 0$. If $a_s:=y^s z \in \BZ[\xi]$ for infinitely many positive $s$, then $ y\in \BZ[\xi]$.

e) One has
\bes O_\xi^2 \sim \begin{cases} r \quad &\text{if $r$ is odd}, \\
r/2 \quad &\text{if $r$ is even}.
\end{cases}
\ees

f) For every integer $0\le k<r$ , $O_\xi$ divides $(\xi;\xi)_{\lfloor k/2\rfloor}\, x_{\bar k}$ .
\end{lemma}

\begin{proof} 
a) Let $c= (a,r) =(b,r)$. Since $1-\xi^c$ divides $1-\xi^a$,  and also $1-\xi^a$ divides $1-\xi^c=1-\xi^{a a^*}$
where $a a^* \equiv c \pmod r$, we have $1-\xi^a \sim 1-\xi^c$. Similarly $1-\xi^b\sim 1-\xi^c$.

Part b) is obtained by substituting $X=1$ into the following identity.
$$
 1+X+\dots+X^{r-1}=\frac{1-X^r}{1-X}=\prod^{r-1}_{i=1} (X-\xi^i)\ .
$$

Part c) follows from the fact that
every Dedekind domain is integrally closed.

d) Let $y = y_1/y_2$ and $z = z_1/z_2$ with $y_1,y_2,z_1,z_2 \in \BZ[\xi]$ and $z_i \neq 0$. Then for infinitely many $s>0$,
$  z_1 y_1^s = a_s z_2 y_2^s$, and hence
\be  (z_1) (y_1)^s = (a_s) (z_2) (y_2)^s,
\label{e403}\ee
where $(x)$ denotes the principal ideal in $\BZ[\xi]$ generated by $x$.
In any Dedekind domain, every ideal
decomposes uniquely into a product of prime ideals:
$$(x)=\prod_i \mathfrak p^{e_i}_i\ $$
and this decomposition respects the multiplication.
From the uniqueness of prime ideal decomposition and \eqref{e403},
we see easily that $y_2\mid y_1$, or $y = y_1/y_2 \in \BZ[\xi]$.

e) First suppose $r$ is odd. Then $ O_\xi = (\xi;\xi)_{\frac{r-1}{2}}$. Since $(1-\xi^k) \sim (1-\xi^{r-k})$
by part (a), we have
$$
O_\xi^2 \sim  (\xi; \xi)_{r-1} = r.
$$

Now suppose $r$ is even. Then $ O_\xi = (\xi;\xi)_{\frac{r-2}{2}}$. Using $(1-\xi^k) \sim (1-\xi^{r-k})$, we have
$$
O_\xi^2 \sim  (\xi; \xi)_{r-1}/ (1-\xi^{r/2}) = r/2
$$
since $\xi^{r/2}=-1$.

f)
First suppose $r$ is odd. Note that for odd $r$,
 $(1- \xi^{j}) \sim (1-\xi^{2j})$ by part (a). One has
\begin{align}
x_{\bar k} = x_{r-k-1} & = \frac{\prod_{j=1}^{r-k-1}(1-\xi^j)}{\prod_{j=1}^{\lfloor \frac{r-k-1}2\rfloor}
(1-\xi^j)} \notag\\
&\sim \frac{\prod_{j=1}^{r-k-1}(1-\xi^j)}{\prod_{j=1}^{\lfloor \frac{r-k-1}2\rfloor}
(1-\xi^{2j})}
 \quad \text{since $(1- \xi^{j}) \sim (1-\xi^{2j})$} \notag \\
& \sim (1-\xi)(1-\xi^3)\cdots (1-\xi^{r-2-2\lfloor\frac{k}2\rfloor})\ . \label{e3001}
\end{align}
Using $(1-\xi^j) \sim (1-\xi^{2j}) \sim (1-\xi^{r-2j})$, we have
\be (\xi;\xi)_{\lfloor\frac k2\rfloor} = \prod_{j=1}^{\lfloor\frac k2\rfloor}
 (1- \xi^{j})  \sim \prod_{j=1}^{\lfloor\frac k2\rfloor} (1-\xi^{r-2j}). \label{e3002}
\ee
Multiplying \eqref{e3001} and \eqref{e3002}, we get
\be
 (\xi;\xi)_{\lfloor\frac k2\rfloor}\,  x_{\bar k} \sim \prod_{j=1}^{(r-1)/2}(1-\xi^{2j-1}) \sim O_\xi\ ,
 \label{e3006}
 \ee
where the second $\sim$ follows from the fact that $1-\xi^{r-a}\sim 1-\xi^a$ for any integer $a$.

Now suppose $r$ is even.
\begin{align}
\frac{O_\xi}{(\xi; \xi)_{\lfloor\frac k2\rfloor}}& = (1-\xi^{r/2-1}) (1-\xi^{r/2-2}) \cdots
(1-\xi^{\lfloor\frac k2\rfloor+1}) \notag \\
& =(1+\xi)(1+\xi^2)\cdots (1+\xi^{\frac r2-1-\lfloor\frac k2\rfloor}). \label{e3003}
\end{align}
On the other hand there exists $f\in\Zx$ such that
\begin{align} x_{\bar k}=x_{r-k-1} &= \frac{\prod_{j=1}^{r-k-1}(1-\xi^j)}
{\prod_{j=1}^{\lfloor\frac {r-k-1}2\rfloor}(1-\xi^j)} \notag\\
&= f\ \frac{\prod_{j=1}^{\lfloor\frac {r-k-1}2\rfloor}(1-\xi^{2j})}
  {\prod_{j=1}^{\lfloor\frac {r-k-1}2\rfloor}(1-\xi^j)} \notag\\
&=f\ (1+\xi)(1+\xi^2)\cdots (1+\xi^{\lfloor\frac {r-k-1}2\rfloor}). \label{e3004}
\end{align}
Note that ${\lfloor\frac {r-k-1}2\rfloor} = \frac r2-1-\lfloor\frac k2\rfloor$ for even $r$.
Compare \eqref{e3003} and \eqref{e3004}, we see that $O_\xi$ divides $(\xi;\xi)_{\lfloor\frac k2\rfloor}
\, x_{\bar k}$.
\end{proof}

\subsection{Quadratic Gauss sums}\label{sec3.3}
For arbitrary integers  $b$ and $d$, the quadratic Gauss sum is defined as
\begin{equation*}
G(b,d,\xi):=\sum^{\ord(\xi)-1}_{n=0} \xi^{b n^2+d n} 
\ .
\end{equation*}
The following is well-known. 

\begin{prop}\label{lGauss}
a) Let $r=\ord(\xi)$ and $c= \gcd(b,r)$. Then
$$
G(b,d,\xi) = \begin{cases}
c\, G(b/c, d/c, \xi^c) & \textrm{ if }\ c\,|\,d \ ;\\
0 & \textrm{ otherwise }.
\end{cases}
$$

b) Suppose $b$  and $r$ are co-prime. Then
 $$ G(b,0,\xi)^2 \sim   \begin{cases}   r  & \text{ if $r$ is odd}\ ; \\
   0 & \text{ if $r\equiv 2 \pmod 4$}\ ; \\
   2r & \text{ if $r\equiv 0 \pmod 4$}\ .
   \end{cases}
   $$
Furthermore  $ G(b,b,\xi)^2 = 2r $ if $ r \equiv 2 \pmod 4$.

c) Suppose $d$ is odd and $r \equiv 0 \pmod 4$. Then $G(b,d,\xi)=0$.

d) Suppose $r_1$ and $r_2$ are co-prime and $r = r_1 r_2$. Then
\be\label{chin}
G(b,d,\xi)=G(b r_1,d,\xi^{r_1})\, G(b r_2,d,\xi^{r_2}).\ee
\end{prop}

\begin{proof}
Part (a) is clear from the definition when $c\,|\,d$. Now suppose that $c\nmid d$. We have
\bas
G(b,d,\xi) & = \sum_{t=0}^{r/c-1} \sum_{s=0}^{c-1} \xi^{b(sr/c+t)^2 + d (sr/c + t)} \\
& = \sum_{t=0}^{r/c-1} \xi^{bt^2+dt} \sum_{s=0}^{c-1} \xi^{sdr/c} = 0\ ,
\end{align*}
where the last equality follows from the fact that $\xi^{dr/c}\ne 1$ and its order divides $c$.

b) After a Galois transformation of the form $ \xi \to \xi^a$, with $a$ co-prime to $r$, one can assume that $b=1$ and $\xi = \exp( 2\pi i/r)$. The result now follows e.g. from  \cite[Chapter 2]{Davenport}.


c) One has
\bas b\left(n + \frac r2\right)^2 + d\left(n + \frac r2\right) & = bn^2 + bnr + br \frac r4 + dn + d\frac r2 \\
&\equiv bn^2 + dn + \frac r2 \pmod r.
\end{align*}

Hence

$$ G(b,d,\xi)= \sum_{n=0}^{r-1} \xi^{bn^2 + dn} = \sum_{n=0}^{r-1} \xi^{b(n+ r/2)^2 + d(n+r/2)}= \xi^{r/2} \, \sum_{n=0}^{r-1} \xi^{bn^2 + dn} = - G(b,d,\xi).$$
It follows that $G(b,d,\xi) =0$.

d) The proof follows easily from the fact that the map $\BZ/r_1 \times \BZ/r_2 \to \BZ/(r_1 r_2)$, defined by
$ (n_1,n_2) \to r_2 n_1 + r_1 n_2$, is an isomorphism.
\end{proof}

\subsection*{Proof of Lemma \ref{lem12}}
Now we are in position to see that
$F^G_{U^+}(\xi)=0$ if and only if $G=SU(2)$ and $\xi^{1/4}$ has order $2r$.

By completing the squares
 we have
\bas
 & (1-\xi) F^G_{U^+}  \sim \frac{2}{1-\xi}
 \sumx^{\xi,G} q^{\frac{n^2-1}4}  (1-q^n) \sim \frac{2}{1-\xi}\left( \sumx^{\xi,G} q^{\frac{n^2-1}{4}}-
\xi^{-1}\sumx^{\xi,G} q^{\frac{(n+2)^2-1}4} \right)
 \\
 & \sim  \begin{cases}  G(4^*,0,\xi) & \text{ if  $G=SO(3)$ ;}\\
   \frac 12 G(1,0, \xi^{1/4}) \quad & \text{ if  $G=SU(2)$ and $\ord(\xi^{1/4})=4r$ ;}\\
    G(1,0, \xi^{1/4}) \quad & \text{ if  $G=SU(2)$ and $\ord(\xi^{1/4})=2r$ ;}\\
    2 G(1,0, \xi^{1/4}) \quad & \text{ if  $G=SU(2)$ and $\ord(\xi^{1/4})=r$ .}
 \end{cases}
\end{align*}
Note that for $G=SO(3)$, the sum is over odd $n$'s, so $n^2-1$ is always divisible by $4$. Hence, 
for any choice of $\xi^{1/4}$, we have $\xi^{(n^2-1)/4}=
\xi^{4^*(n^2-1)}$
with $4^*4=1 \pmod r$.

If $r$ is even, then $\ord(\xi^{1/4})$ is always $4r$.
Now  Proposition \ref{lGauss} (b) implies the claim.
\qed




\subsection{Simplification of $H^G(k,b, \ve)$}\label{simp}

\begin{lemma}\label{h1}
a)  For integers $k,b$, and $\ve \in \{0,1\}$,  there is $f_\ve(z,q) \in I_{2k+1+\ve}$ such that
$$
H^{G}(k, b,\ve)  \sim \frac2{x_{2k+1+\ve}}\sumx^{\xi,G} q^{\frac{ b (n^2-1)}4 - \frac{3\ve n}2}  f_\ve(q^n,q)\ .
$$
More precisely, one can choose $f_\ve= z^{-k}\, (q^{-k}z; q)_{2k+1+\ve}$.

b) One has $\sqrt 2, \sqrt r \in \Z[\xi^{1/4}, e_8]$ and
$$ H^G(0,\pm 1,0) \sim  
\begin{cases} \sqrt r  \quad & \text{ if  $G=SO(3)$ ;}\\
  \sqrt{2r} & \text{ if  $G=SU(2)$ and $\ord(\xi^{1/4})=4r$ ;}\\
    2 \sqrt r \quad & \text{ if  $G=SU(2)$ and $\ord(\xi^{1/4})=r$ .}
\end{cases}
$$

c) One has $\cD^G \in \Z[\xi^{1/4}, e_8]$ and $(1-\xi) \cD^G \sim H^G(0,\pm1,0)$.

d)  Suppose $b$ and $r$ are even. 
Then $H^{SU(2)}(k,b,1)=0$.
\end{lemma}

\begin{proof}
a)   We will use the following simple observation: For  $g(z,q) \in \BQ[z^{\pm 1/2}, q^{\pm 1/4}]$, we have
\be\label{siSU}
 \sumx^{\xi, G}\ q^{\frac{ b (n^2-1)}4} g(q^{n/2},q) =
\sumx^{\xi, G} q^{\frac{ b (n^2-1)}4} g(q^{-n/2},q).
\ee
To prove it, one only needs to consider $g(z,q)=z^{a/2}, a\in\Z$. 
Then 
$$\text{LHS} = \sumx^{\xi, G} \;q^{\frac{ b (n^2-1)}4+\frac{an}2} = {\sum_{n\to -n}}^{\xi, G} \; q^{\frac{ b (n^2-1)}4 - \frac{an}2} = \text{RHS}\ .$$

One can check that $\cb n\prod_{j=-k}^k\cb{n+j} = (q^{-kn-n}-q^{-kn}) (q^{n-k}; q)_{2k+1}$.
Then  we get
\begin{align*}
\ex \left(\frac{\cb{2k+1}!}{\cb k!}\right) H^{G}(k, b,0) &=
\sumx^{\xi, G} q^{\frac{ b (n^2-1)}4}
\cb n\prod_{j=-k}^k\cb{n+j} \nonumber\\
 & =- 2 \sumx^{\xi,G} q^{\frac{ b (n^2-1)}4}  q^{-nk} (q^{n-k}; q)_{2k+1}  \ ,
\end{align*}
where the last equality follows from  \eqref{stand},
\eqref{siSU} and the fact that 
$$q^{-kn-n} (q^{n-k}; q)_{2k+1}= -q^{-kn}(q^{n-k}; q)_{2k+1}\mid_{n\to -n} \ .$$

Analogously, we have
\begin{align*}
\ex \left(\frac{\cb{2k+1}!}{\cb k!} \lam_{k+1}\right) H^{G}(k, b,1) &=
\sumx^{\xi, G} q^{\frac{ b (n^2-1)}4}
\cb n \, \lam_n  \prod_{j=-k}^k\cb{n+j} \nonumber\\
 & =- 2 \sumx^{\xi,G} q^{\frac{ b (n^2-1)}4}  q^{-n(k+3/2)} (q^{n-k}; q)_{2k+2} \ .
\end{align*}
This proves a).

b) Let us first show that $\sqrt r,\sqrt 2 \in \Z[\xi^{1/4}, e_8]$.
Observe that $\sqrt 2 \in \Z[e_8]$. Further,
$$ \prod^{(r-1)/2}_{j=1} |1-\xi^j| = 
 \begin{cases} \sqrt r & \text{ if $r$ is odd;}\\
 \sqrt{r/2} & \text{ if $r$ is even.}
 \end{cases}
$$
Since $|1-\xi^j|=\pm\sqrt{-1}\,(\xi^{j/2}-\xi^{-j/2})$,
we have $\sqrt r\in \Z[\xi^{1/4}, e_4]$.

Part (b)  follows now from \eqref{tri}, Proposition \ref{lGauss} (b) and the 
proof of Lemma \ref{lem12}.

c) 
Since $\cD_G:=|F^G_{U^+}|$, from the proof of Lemma \ref{lem12}, we get
$$|1-\xi|\cD^G      
=  \begin{cases}  \sqrt r & \text{ if  $G=SO(3)$ ;}\\
   \sqrt {2r}\quad & \text{ if  $G=SU(2)$ and $\ord(\xi^{1/4})=4r$ ;}\\
   2 \sqrt r\quad & \text{ if  $G=SU(2)$ and $\ord(\xi^{1/4})=r$ .}
 \end{cases}
 $$
Clearly, $\sqrt r$ is divisible by $|1-\xi|$, so $\cD^G \in \Z[\xi^{1/4}, e_8]$.
The second statement follows from (b).

d) By part (a), it is enough to show that 
$$\sumx^{\xi,G} q^{\frac{ b (n^2-1)}4 - \frac{3n}2}  f(q^n,q)=0$$
for any $f \in \Zzq$. We can assume $f=z^a, a\in \Z$.  Assume $b= 2b'$. Then

\begin{align*}
4 \sumx^{\xi,SU(2)} q^{\frac{ b (n^2-1)}4 - \frac{3n}2}  q^{na} =
  2 \xi^{-b'/2}  \sum_{n=0}^{2r-1}  \left (\xi^{1/2}\right)^ {b' n^2 -3n + 2na}= 
2 \xi^{-b'/2}  \, G(b', 2a-3, \xi^{1/2}),
\end{align*}
which is 0 by Proposition \ref{lGauss} (c), since $\ord(\xi^{1/2})$ is always $2r$ if $r$ is even.
\end{proof}


\subsection{Lens spaces} 
Suppose $L\sqcup L'$ is the Hopf link with framing $b\neq 0$ on $L$ and framing 0 on $L'$. Besides, the color of $L'$ is a fixed number $a$.
By surgery on $L$ from $(S^3,L')$ we get the pair $(\L({b,1}),L')$,
where $\L(b,1)$ is the lens space.
It is known that $J_{L\sqcup L'}(n) = q^{b(n^2-1)/4}\, [na]$. Hence we have
\be \label{17}
\tau^G_{\L({b,1}),L'}(\xi)= \frac{{\sum_n}^{\xi,G} q^{b (n^2-1)/4} [na][n]}{{\sum_n}^{\xi,G} q^{\sn(b) (n^2-1)/4} [n]^2}\ .
\ee

Note that the invariant of $\L(b,-1)=\L(-b,1)$ is just a complex conjugate of
\eqref{17}.

\begin{lemma}\label{lens} a) If $b$ and $r$ are co-prime, 
then $\tau^{SO(3)}_{\L({b,1})}$ is invertible in $\Z[\xi]$.

b) Suppose $r$ is even. For $b=2^k$,
 there is a knot $K$ in the lens space $M= \L(2^k,-1)$ colored by an odd number  such that
$$\tau^{SU(2)}_{M, K}(\xi)\not= 0\, .$$
\end{lemma}

\begin{proof} a) The  $SO(3)$
invariant of $\L(b,1)$ can be easily computed. By completing the square we have
$$\tau^{SO(3)}_{\L(b,1)}(\xi)\;=\;
\xi^{(\sn(b)-b)/4}\;
\frac{(1-\xi^{-b^*})}{(1-\xi^{-1})} \;  \frac{G(b,0,\xi)}{G(1,0,\xi)}\, ,$$
which is a unit in $\Z[\xi^{1/4}]$ by Proposition \ref{lGauss} (b). Here
$b^*b\equiv 1\pmod r$. 

b) Let $L\sqcup L'$ be the Hopf link with framing $-b=-2^k$ on $L$ and framing $0$ on $L'$.
 Suppose $L'$ is colored by $a=2s+1$.
Surgery on $L$ gives us a pair $(M, K)=(\L(2^k,-1),K)$.


An easy calculation shows
\be
 \tau^{SU(2)}_{M, K}(\xi) \sim \frac{G(-b,4s+4, \xi^{1/4}) - G(-b,4s, \xi^{1/4}) }{(1-\xi)\, G(-1,0, \xi^{1/4})}\ .
 \label{e338}
 \ee

For $b=2$, then $M =\R P^3$.  Choose $s=0$, or $a=1$. Then  $\tau_{M,K}^{SU(2)}(\xi)= \tau_M^{SU(2)}(\xi)\neq 0$.

For $b=4$ again choose $s=0$. Then  one and only one term in the numerator of \eqref{e338} is  zero, by Proposition \ref{lGauss}.

Suppose $b=2^k>4$. Then $c:=(b,4r)>4$.
Choose $s$ such that $G(-b, 4s, \xi^{1/4})\neq 0$ (see Proposition \ref{lGauss}). Then $c \mid 4s$, and $c$ does not  divide $4s+4$. Hence
$G(-b, 4s+4, \xi^{1/4})= 0$. We conclude that $\tau^{SU(2)}_{M, K}(\xi)\neq 0$.
\end{proof}

\section{Symmetry Principle and splitting of the $SU(2)$ invariant}\label{sys}
The symmetry principle of the colored Jones polynomial and the splitting of the $SU(2)$  WRT
invariant were discovered by Kirby and Melvin in \cite{KM}. In \cite{Le_Duke, Le4}, the third author generalized these
to all higher ranked Lie groups.
Here we extend  the symmetry principle and splitting to the case
of pairs of a 3-manifold and a colored link inside.  
We show  that the symmetry principle for a link in an arbitrary 3-manifold 
holds only  for $SO(3)$ invariant, but does not hold for the $SU(2)$ invariant.
\subsection{Symmetry Principle for links in $S^3$}

\def\al{\alpha}

Suppose $\xi$ is a root of unity of order $r$. Then the colored Jones polynomial at $\xi$ is periodic  with period $2r$, i.e.
\be\label{s1} 
\ev_\xi \big( J_L( n_1, \dots, n_i + 2r, \dots, n_m) \big) = \ev_\xi \big( J_L( n_1,\dots, n_i, \dots, n_m) \big)\ ,
\ee
and under the reflection $r-n \to r+n$ it behaves as follows: 
\be\label{s2}
 \ev_\xi \big( J_L(n_1,\dots, r+n_i, \dots, n_m) \big) = -
 \ev_\xi \big( J_L(n_1,\dots, r-n_i, \dots, n_m) \big)\ 
\ee
(see \cite{Le_Duke}).
This means that one can restrict the colors to the interval $[0,r]$.

The symmetry principle tells us how $J_L$ behaves under  the transformation $n \to r-n$.
More precisely, let
 $\BZ/2 = \{0,1\}$ act on $\BZ/r\Z$ by $0 * n = n$, $1*n = r-n$. 
For $\a = (a_1,\ldots,a_m) \in \{0,1\}^m$ and $\bn = (n_1, \ldots, n_m) \in (\Z/r\Z)^m$,
let $\a * \bn = (a_1*n_1,\ldots, a_m*n_m)$. In addition, we set
$\hat n := n-1$ for any integer $n\in \Z$.

\begin{prop}\label{lsy}
 Suppose $(\ell_{ij})$ is the linking matrix of $L$. With the notations as above one has 
\bes
\ev_\xi \big( J_L( \a * \bn )\big) = \left (-\xi^{r/2}\right )^{\sum_i a_i}\, \xi^t \ex \big( \,  J_L( \bn ) \big)\ ,
\ees
where
\be \label{ttt}
t= \frac{r(r-2)}4 \sum_{i,j} \ell_{ij} a_i a_j + \frac r 2 \sum_{i,j} \ell_{ij} a_i \hat n_j\ ,
\ee
and $(\ell_{ij})$ is the linking matrix of $L$.
\end{prop}

\begin{proof} This is the $sl_2$ case of \cite[Theorem 2.6]{Le_Duke}. The factor $\left (-\xi^{r/2}\right )^{\sum_i \al_i}$ comes from the difference between our $J_L$
and $Q_L$ in \cite{Le_Duke}, where $Q_L$ is equal to $J_L$ times the quantum dimensions of the colors on $L$.
\end{proof}

\begin{rmk}
If $\ord(\xi^{1/2})= 2r$, then $-\xi^{r/2}=1$, and  this case was considered in \cite{KM}. Proposition \ref{lsy} handles also 
the case when $\ord(\xi^{1/2})\neq 2r$, i.e. $\ord(\xi^{1/2})= r$.
\end{rmk}

A simple but useful
observation is that if all entries of $\bn$ are odd, 
then the second term in \eqref{ttt} 
is an integer multiple of $r$, hence can be removed.

\subsection{WRT $SO(3)$ invariant for an arbitrary colored link in $M$}\label{ext}
In the literature, the WRT $SO(3)$ invariant of the pair $(M, L')$ was defined in the case
when all colors of $L'$ are odd or all equal to 2 (compare \cite{MR}).
Here we extend this definition to arbitrary colors. 
Since the colors of $L'$ will play an important role in this section, we
will make the dependence on them explicit in the notation.

Note that $SO(3)$ invariants of $M$ with  evenly colored links inside are
 not coming from  Topological Quantum Field
Theories. The main reason is that fusion preserves odd colors.
However,  fusion of an odd and an even color produce
 an even color. This violates the invariance of $\eqref{inv}$ under sliding 
in the case when some of the $s_i$'s are even. We will  show that this defect
can easily be taken into account
by a simple factor depending on the linking matrix and parity of the colors only.

Throughout the remaining of this section let $r=\ord(\xi)$ be odd 
and $\bs = (s_1,\ldots, s_l)$ be the color on $L'$.
Let $(\ell_{ij})$ and $(p_{ij})$ be the linking matrices of $L$ and $L'$ respectively.
The linking number between the $i$-th component of $L$ and the $j$-th component of $L'$ will be denoted
by $\tilde\ell_{ij}$.

Let
\be\label{sta-SO(3)}
F^{SO(3)}_{L\sqcup L'}(\xi; \bs) : = \xi^{\mu{(L',\bs)}}
\sumxi^{\xi,SO(3)} [\bn]\ J_{L\sqcup L'}(\bn, \bs)\ ,
\ee
where $[\bn] := \prod_{i=1}^m[n_i]$ and
\bes\label{mu}
 \mu (L'; \bs):= -\frac{r(r-2)}4 \sum_{i,j=1}^l p_{ij}\hat s_i\hat s_j\ .
\ees

Observe that, when all $s_i$ are  odd, \eqref{sta-SO(3)} coincides with \eqref{sta}.

\begin{lemma}\label{2nd}
$F^{SO(3)}_{L\sqcup L'}(\xi; \bs)$ is invariant under the handle slide of a component of $L$ or $L'$ over a component of $L$.
\end{lemma}

\begin{proof}
The invariance under
sliding of one component of $L$ over another component of $L$
 follows by standard arguments (see e.g.  \cite{Lickorish}).


Let $L\sqcup L''$ be the link obtained from $L\sqcup L'$
by sliding a component of $L'$ over a component of $L$. It is enough
 to show that
$$
F^{SO(3)}_{L\sqcup L'}(\xi; \bs) = F^{SO(3)}_{L\sqcup L''}(\xi; \bs)\, .
$$
Using the fact that $\hat s *\hat s* s=s$ for any $s$ with
$\hat s\equiv s-1\pmod 2$
and Proposition \ref{lsy}, we have
\be\label{2223}
 \xi^{\mu(L',\bs)}\;J_{L\sqcup L'}(\bn,\bs)=(-\xi^{r/2})^{\sum_i \hat s_i} \;J_{L\sqcup L'}(\bn, \hat \bs *\bs)\, .\ee
Here we used the fact $\hat s *s$ is always odd, and hence 
all summands of $\sum p_{ij}\hat s_i \widehat {\hat s_j * s_j}$ are even.

By the invariance of $F^{SO(3)}_{L\sqcup L'}(\xi,\bs)$ in the standard case when
all colors are odd, we get
\be\label{225}\sumxi^{\xi, SO(3)}[\bn] J_{L\sqcup L'}(\bn, \hat \bs *\bs)=
\sumxi^{\xi, SO(3)}[\bn] J_{L\sqcup L''}(\bn, \hat \bs *\bs)\, .\ee

Further using Proposition \ref{lsy} again, we obtain
\be\label{221}
J_{L\sqcup L''}(\bn, \hat \bs *\bs)=
(-\xi^{r/2})^{\sum_i \hat s_i}\; \xi^{\mu(L'',\bs)}J_{L\sqcup L''}(\bn, \bs)\, .\ee
Inserting \eqref{2223}, \eqref{221} into \eqref{225} we get the result.
\end{proof}

Lemma \ref{2nd} suggests to define $\tau^{SO(3)}_{M,L'}(\xi; \bs)$ for arbitrary $\bs$ by substituting
  $F^{SO(3)}_{L\sqcup L'}(\xi)$ given by  $\eqref{sta-SO(3)}$
 into  $\eqref{inv}$. When all colors of $L'$ are odd, 
the only additional factor $\xi^{\mu(L',\bs)}$ is 1 and
we get back our old invariant.

\begin{cor}\label{sinv}
$\tau^{SO(3)}_{M,L'}(\xi; \bs)$  is an invariant of the pair $(M,L')$.
\end{cor}

\begin{rmk}\label{dd}
For a colored link $L$ in the 3-sphere, our invariant equals to 
$$\tau^{SO(3)}_{S^3,L}(\xi; \bs) = \xi^{-r(r-2)/4\; 
\sum _{i,j} l_{ij}\hat s_i\hat s_j} \ex(J_L(\bs))\, .
$$
Hence if some colors of $L$ are even, this invariant might differ
from the colored Jones polynomial by some factor depending on the
linking matrix $(l_{ij})$ of $L$.

\end{rmk}

\subsection{Symmetry Principle for the WRT $SO(3)$ invariant}
We use the same notations as in the previous section.

\begin{prop}\label{p007}
For $\a\in \{0,1\}^l$ and $\bs\in (\BZ/r\Z)^l$  one has
\bes
\tau^{SO(3)}_{M,L'}(\xi; \a*\bs) = \left (-\xi^{r/2}\right )^{\sum_i a_i}\,   \tau^{SO(3)}_{M,L'}(\xi; \bs)\ .
\ees
\end{prop}

\begin{proof} 
By \ref{lsy} we have
$$
\sumxi^{\xi, SO(3)} [\bn] J_{L\sqcup L'}(\bn, \a*\bs) = (-\xi^{r/2})^{\sum_i a_i} \xi^u \sumxi^{\xi, SO(3)} [\bn]
J_{L\sqcup L'}(\bn, \bs)\ ,
$$
where $u=\frac{r(r-2)}4 \sum_{i,j} p_{ij} a_i a_j+ \frac r2\sum_{i,j} p_{ij}a_i\hat s_j$. 
Here we use the fact that $\bn$ is odd in the above sum.
On the other hand
$$
w := 
\mu(L'; \bs) - \mu(L'; \a*\bs) = \frac{r(r-2)}4 \sum_{i,j} p_{ij}(\widehat{a_i*s_i}\widehat{a_j*s_j} - \hat s_i \hat s_j)\ .
$$
Then
\bas
& \qquad u - w  \equiv \frac{r(r-2)}4\sum_{i} p_{ii} (\hat s_i^2 - \widehat{a_i*s_i}^2 + 2 a_i\hat s_i+ a_i^2) \\
& + \frac r2 \sum_{i<j} p_{ij}(\hat s_i\hat s_j- \widehat{a_i*s_i}\widehat{a_j*s_j} + a_i\hat s_j + a_j\hat s_i +a_ia_j)
 \quad\equiv 0 \mod r\ ,
\end{align*}
which can be verified directly. 
\end{proof}


\begin{rmk}
Proposition \ref{p007} is not true for the WRT $SU(2)$ invariant.
For example, consider the Hopf link with framing $2$ on the first component and framing $0$ on the second.
Surgery on the first component produces a pair $(\R P^3,K)$. If $\ord(\xi)=3$ and 
$\ord(\xi^{1/4})=12$ then $\tau^{SU(2)}_{M,K}(1;\xi)=0$
and $\tau^{SU(2)}_{M,K}(1*1=2;\xi)\ne0$.
\end{rmk}

\subsection{Splitting} In \cite{KM} it was proved that when both the $SO(3)$ and $SU(2)$ WRT invariants can be defined, i.e. when $r$ is odd, then one has the splitting
$$ \tau_M^{SU(2)}(\xi) = \tau^{\BZ/2}_M (\xi) \,  \tau_M^{SO(3)}(\xi)\ ,$$
where $\tau^{\BZ/2}_M (\xi)$ is a simple invariant depending only on the linking pairing of $M$. Here we generalize this result for invariants of pairs $L'\subset M$. We will follow the
approach in \cite{Le2}, where the splitting is generalized to all higher ranked
 simple Lie algebras.

Let $s_1,\dots,s_l$ be the colors on $L'$ and set
\be
 F_{L\sqcup L'}^{\BZ/2}(\xi; \bs) = \xi^{\frac{r(r-2)}4 \sum p_{ij} \hat s_i  \hat s_j }\!\!\!\!\!\!\!\!
  \sum_{\al_1,\dots,\al_m \in \{0,1\}} 
\xi^{\frac{r(r-2)}4 \sum \ell_{ij} \al_i \al_j + \frac r2 \sum \ve_i \,  \al_i   }\ ,
 \label{e900}
 \ee
where $(\ell_{ij})$ and $(p_{ij})$ are the linking matrices of $L$ and $L'$ respectively,
and $\ve_i$ is defined by \eqref{3306}.

For example
\be F_{U^{\pm}}^{\BZ/2}(\xi)= 1 + \xi^{\pm \frac{r (r-2)}{4}}\ .
\label{e9007}
\ee

We will assume that  $r=\ord(\xi)$ is odd and $\xi^{1/4}$ is chosen so that $\ord(\xi^{1/4}) \neq 2r$, i.e. $\ord(\xi^{1/4})$ is either $r$ or $4r$. This choice guarantees that 
$F_{U^{\pm}}^{\BZ/2}(\xi) \neq 0$.
Define
\be
\tau^{\BZ/2}_{M,L'} (\xi;\bs) = \frac{ F_{L\sqcup L'}^{\BZ/2}(\xi;\bs) }{\left( F_{U^{+}}^{\BZ/2}(\xi)\right)^{\beta_+} \, \left( F_{U^{-}}^{\BZ/2}(\xi)\right)^{\beta_-}\, \left | F_{U^{+}}^{\BZ/2}(\xi)\right|^{\beta}}\ .
\label{e901}
\ee
 Then $\tau^{\BZ/2}_{M,L'} (\xi;\bs)$ is an invariant of the pair $(M,L')$.
 
\begin{rmk}
This type of
invariants were studied in \cite{MOO, De1} for 3-manifolds without links inside, and in \cite{De2}
for 3-manifolds with links inside.
When the abelian group is $\Z/2\Z$,  set the parameters $c_i$  in \cite{De2}
to be equal to $s_i-1 \mod 2$, and define
the quadratic form $q$ on $\Z/2\Z$
as follows: $q(0)=0$, $q(1)=(r-2)/4$, then the invariant introduced in \cite{De2}
 is equal to $\tau^{\BZ/2}_{M,L'} (\xi; \bs)$ after setting  $\xi^{r/4} = \sqrt{-1}$. 
\end{rmk}


\begin{prop}\label{sp} Suppose $r=\ord(\xi)$ is odd, and $\ord(\xi^{1/4})$ is either $r$ or $4r$. 

(a) One has the splitting
$$ \tau_{M,L'}^{SU(2)}(\xi; \bs) = 
\tau^{\BZ/2}_{M,L'} (\xi; \bs) \,  \tau_{M, L'} ^{SO(3)}(\xi; \bs)\, .$$

(b)  If $\ord(\xi^{1/4})=r$, then $\tau^{\BZ/2}_{M,L'} (\xi; \bs)=1$ and $$\tau_{M,L'}^{SU(2)}(\xi; \bs) =
\tau_{M, L'} ^{SO(3)}(\xi; \bs)\ .$$

(c)  One has the integrality $$\tau^{\BZ/2}_{M,L'} (\xi; \bs)\in  \Z[\xi^{1/4},e_8]\ .$$
\end{prop}

\begin{proof} 
(a) Recall that $\tilde \ell_{ij}$ is the linking
number between the $i$-th component of $L$ and the $j$-th component of $L'$. 
Also note that
$\forall \a\in (\Z/2\Z)^m$, $\a*(\a*\bn)=\bn$.
By Proposition \ref{lsy} we have
$$
 \ex ([\bn] J_{L\sqcup L'}(\bn, \bs))
= \xi^t 
\ex ([\a * \bn]J_{L\sqcup L'}(\a * \bn, \bs)) \ ,
$$
where $t=\frac{r(r-2)}4 \sum \ell_{ij} a_i a_j + \frac r2 \sum \tilde \ell_{ij} a_i \hat s_j$. 
Note that the factor $(-\xi^{r/2})^{\sum a_i}$ is 
missing because of the quantum integers. Therefore by \eqref{sta}, \eqref{sta-SO(3)} and \eqref{e901} we have
\bas
F_{L\sqcup L'}^{SU(2)}(\xi; \bs) & = \xi^{\frac{r(r-2)}4
\sum p_{ij}\hat s_i\hat s_j} \sum_{a_i \in\cb{0,1}} \xi^{\frac{r(r-2)}4\sum \ell_{ij} a_i a_j
+ \frac r2\sum 
\tilde \ell_{ij} a_i \hat s_j} \ F_{L\sqcup L'}^{SO(3)}(\xi; \bs) \\
& = F_{L\sqcup L'}^{\Z/2}(\xi; \bs)\ F_{L\sqcup L'}^{SO(3)}(\xi; \bs)\ , 
\end{align*}
which implies (a).

(b) If $\ord(\xi^{1/4})=r$, then by \eqref{e900}, $F_{L\sqcup L'}^{\BZ/2}(\xi; \bs)= 2^m$. In particular, $F_{U^{\pm}}^{\BZ/2}(\xi; \bs)=2$. It follows that $\tau^{\BZ/2}_{M,L'} (\xi; \bs)=1$.

(c) The case $\ord(\xi^{1/4})=r$ was covered by (b). Assume that $\ord(\xi^{1/4})= 4r$. Then from \eqref{e9007} it follows that $ F_{U^{\pm}}^{\BZ/2}(\xi) \sim \sqrt 2$.
Hence the denominator of \eqref{e901} is $\sim (\sqrt 2)^m$.

According to \cite[p. 522]{KM}, 
we may
assume $\ell_{ij}\equiv 0 \mod 2$ if $i\ne j$. Since $\ell_{ij} \al_i\al_j$ appears twice in the exponent in \eqref{e900} if $i\ne j$, we can write
\bas
F_{L\sqcup L'}^{\BZ/2}(\xi; \bs) & \sim  \prod_{i=1}^m \left(
\sum_{\al_i \in \{0,1\}} \xi^{\frac14 r(r-2) \ell_{ii} \al_i^2 + \frac r2 \ve_i \,  \al_i   }
\right)\\
& = \prod_{i=1}^m \left(  1+  \xi^{\frac14 r(r-2) \ell_{ii}  + \frac r2 \ve_i \,    }   \right)\ .
\end{align*}
Since
$ \xi^{\frac14 r(r-2) \sum \ell_{ii}  + \frac r2 \sum \ve_i \,    }$ is a 4-th root of unity, 
each factor in the above product is
either 2, 0, or $\sim \sqrt 2$, and hence is divisible by $\sqrt 2$. This means $F_{L\sqcup L'}^{\BZ/2}(\xi; \bs)$, 
the numerator of \eqref{e901}, is divisible by $(\sqrt 2)^m$, and the statement follows.
\end{proof}

\section{Diagonalization of 3-manifolds}\label{secd} We recall and refine some well-known facts about diagonalization of 3-manifolds. The first diagonalization result was obtained in
\cite{Ohtsuki} and was further developed in \cite{Le4,bl,bbl1}.

A 3-manifold is said to be {\em diagonal of prime type}  if it
can be obtained by surgery along a framed link $L \subset S^3$ with diagonal linking matrix $\diag(b_1,\dots,b_m)$ such that $b_i= \pm p_i^{e_i}$, where each $p_i$ is a prime, 1, or 0.
Denote by $\L(b,a)$ the lens space obtained from $S^3$ by surgery on the unknot with framing $b/a$. Also $M \#M'$ is the connected sum of $M$ and $M'$ and
$M^{\#s}$ is the connected sum of $s$ copies of $M$.

\begin{prop} For every  3-manifold $M$, there exists a 3-manifold $N$ of the form
$$ N = \L(2^{k_1},-1) \# \cdots \# \L(2^{k_j},-1)\ ,$$
such that for every positive integer $s$, $M^{\# 2s}\# N $ is diagonal of prime type.
\label{p402}
\end{prop}
To prepare for the proof we recall some well-known facts about linking pairing.
{\em A linking pairing} on a finite abelian group $G$ is a non-singular symmetric bilinear map
from $G\times G$ to $\BQ/\BZ$. Two linking pairings $\nu, \nu'$  on respectively $G,G'$ are {\em isomorphic} if there is an isomorphism
between $G$ and $G'$ carrying $\nu$ to $\nu'$. With the obvious block sum, the set of equivalence classes of linking
pairings is a semigroup.

Any non-singular $n \times n$ symmetric matrix $B$ with integer entries gives rise to a linking
pairing $\phi_B$ on $\BZ^n /B \BZ^n$ by
$\phi_B(x,x') =  x^t B^{-1} x' \in \BQ/\BZ$, where $x, x' \in \BZ^n$ and $x^t$ is the transpose of $x$. A linking pairing is {\em diagonal of type $B$} if it is isomorphic to $\phi_B$, where
$B$ is a non-singular $n \times n$ diagonal matrix  with integer entries.

{\em An enhancement} of an $n \times n$ symmetric matrix $B$ is any matrix of the form $B \oplus D$, where $D$ is a diagonal matrix with entries 0 or $\pm 1$ on the diagonal.

For any closed oriented 3-manifold $M$, there is a
linking pairing $\phi(M)$ on the torsion subgroup of $H_1(M,\BZ)$ defined by
the Poincare duality, see \cite{KK}. For example, if $b\neq 0$ is an integer, then the lens space $\L(b,1)$ has linking pairing $\phi_{(b)}$, and $\L(b,-1)$ has linking pairing $\phi_{(-b)}$.
Here $(b)$ is the $1\times 1$ matrix with entry $b$.

It is clear that $\phi({M\#M'})= \phi(M) \oplus \phi(M')$. The result of \cite[Section 3.5]{Le4} shows the following.

\begin{prop} \label{p401}
If the linking pairing $\phi(M)$ on the torsion subgroup of $H_1(M,\BZ)$ is diagonal of type $B$, then $M$ can be obtained from $S^3$ by surgery along an oriented framed link whose linking matrix is
an enhancement of $B$.
\end{prop}

\begin{proof}[Proof of Proposition \ref{p402}]
In \cite[Section 2.2]{bl} it was noticed that $\phi(M\#M)$ is almost diagonal. More precisely,
\be \phi({M\#M})= \phi_B \oplus \nu,
\label{e400}
\ee
where $B$ is a diagonal matrix whose diagonal entries are prime powers and $\nu$ has the form
$$ \nu =\bigoplus_{i=1}^j E_0^{k_i}.
$$
Here $E_0^k$ is a certain linking form on $\BZ/2^k \times \BZ/2^k$. We don't need
the exact description of $E_0^k$. For us it is important that (see \cite{KK})
\be
E_0^k \oplus \phi_{(-2^k)}= \phi_{(-2^k)} \oplus \phi_{(2^k)} \oplus \phi_{(2^k)}.
\label{e339}
\ee
Note that there is still one $\phi_{(-2^k)}$ in the right hand side of \eqref{e339}. From \eqref{e339} and \eqref{e400},

\be \phi(N \#(M\#M)^{\#s})= s\, \phi_B \oplus \bigoplus_{i=1}^j 
\left( \phi_{(-2^{k_i})}
\oplus 2s\,  \phi_{(2^{k_i})}\right)= \phi_{B'},\ee
 where $B'$ is a diagonal matrix with diagonal entries of the form $\pm p^m$ with prime $p$. By Proposition \ref{p401}, $N \#(M\#M)^{\#s}$ is diagonal of prime type.
 This completes the proof of Proposition \ref{p402}.
\end{proof}

\section{Proof of the integrality in the $SO(3)$ case}
Throughout this section $G=SO(3)$ and $\xi$ is a root of unity of odd
order $r$. 

\begin{prop}\label{lem}
For  integer $0 \le k \le (r-3)/2$, arbitrary integer $b$, and $\ve\in \{ 0,1\}$,
 \be \frac{H^{SO(3)}(k, b,\ve)}{H^{SO(3)}(0, \pm1,0)}\in\Z[\xi^{1/4}, e_8]\ ,
 \label{e5520}
 \ee
 and
 \be
 \frac{H^{SO(3)}(k, 0,\ve)}{(1-\xi) \, \cD^{SO(3)}}\in\Z[\xi^{1/4}, e_8]\ . \label{e5521}
 \ee
\end{prop}

\begin{proof} First note that by Lemmas \ref{ab} (e) and \ref{h1} (b), $O_\xi\sim H^{SO(3)}(0,\pm1,0)$.
By Lemma \ref{h1} (a), there is $f_\ve(z,q) \in I_{2k+1+\ve}\subset \BZ[z^{\pm 1}, q^{\pm 1}]$ such that
\be
\frac{H^{SO(3)}(k, b,\ve)}{H^{SO(3)}(0, \pm1,0)} \sim \frac{2 \sumx^{\xi,SO(3)} q^{\frac{ b (n^2-1)}4} q^{\frac{-3\ve n}2} f_\ve(q^n,q)}{x_{2k+1+\ve} O_\xi}\ .
\label{e115}
\ee

Since $r$ is odd, $(n^2-1)/4$ and $(1-n)/2$ are integers, and there are integers $2^*, 4^*$ such that $2^* \,  2 \equiv 4^* \, 4 \equiv 1 \pmod r$. We then have
$\xi^{(n^2-1)/4}= \xi^{4^*(n^2-1)}$, $\xi^{-3n/2} = \xi^{-3/2} \xi^{3(1-n)\, 2^*}$.

The numerator of \eqref{e115} is
\begin{align}
2 \sumx^{\xi,SO(3)} q^{\frac{ b (n^2-1)}4} q^{\frac{-3\ve n}2} f(q^n,q) & =  \xi^{-3\ve/2} \sum_{\substack{n = 0 \\ n \text{ odd}}}^{2r-1}
 \xi^{4^* { b (n^2-1)}+ 3\ve  2^* (1-n)}\, f_\ve(\xi^n,\xi) \notag \\
& =\xi^{-3\ve/2} \sum_{n = 0 }^{r-1}
 \xi^{4^* { b (n^2-1)}+ 3\ve  2^* (1-n)}\, f_\ve(\xi^n,\xi) \ , \label{e216}
\end{align}
where the second identity follows by replacing odd $n \in [r,2r-1]$
with $n-r$, which is even and in $[0,r-1]$.

By Proposition \ref{div1}, the right hand side of \eqref{e216} is divisible by the denominator of the right hand side of
\eqref{e115}, and \eqref{e5520} follows.

Statement \eqref{e5521} follows from \eqref{e5520} with $b=0$ and Lemma \ref{h1}(c), which says that  $ {H^{SO(3)}(0, \pm1,0)}\sim (1-\xi) \, \cD^{SO(3)}$.
\end{proof}

\subsection*{Proof of Theorem \ref{main2}}

By Proposition \ref{lem}, each factor in the right hand side  of \eqref{e_diag} is in $\Z[\xi^{1/4}, e_8]$, hence
$\tau^{SO(3)}_{M, L'}(\xi)\in \Z[\xi^{1/4}, e_8]$ if $M$ is diagonal.

Now suppose $M$ is an arbitrary 3-manifold. Let $N$ be the manifold  described in
Proposition \ref{p402}, for which
$ M\# M \# N
$
is diagonal.  Since the WRT invariant is multiplicative with respect to  connected sum,
 we get

$$ \left( \tau^{SO(3)}_{M, L'}(\xi)\right)^2 \, \tau^{SO(3)}_N(\xi) \in \Z[\xi^{1/4}, e_8]\ .$$
Since $2^{k}$ is coprime to $r$, $\tau^{SO(3)}_{\L(2^{k},-1)}(\xi)$ is a unit in $\Zxf$ by Lemma \ref{lens} (a). It follows that
$\tau^{SO(3)}_N(\xi)$ is a unit, hence $\left( \tau^{SO(3)}_{M, L'}(\xi)\right)^2 \in \Z[\xi^{1/4}, e_8]$.
 By Lemma \ref{ab} (c),   $ \tau^{SO(3)}_{M, L'}(\xi) \in \Z[\xi^{1/4}, e_8]$.
 This completes the proof of the theorem.

\qed

\section{Proof of the integrality in the $SU(2)$ case}   \label{secsu2}

If the order of $\xi$ is odd, then by the splitting  property (Proposition \ref{sp}), 

$$ \tau_{M,L'}^{SU(2)}(\xi) = \tau^{\BZ/2}_{M,L'} (\xi) \,  \tau_{M,L'} ^{SO(3)}(\xi)\ .$$
Both factors of the right hand side are algebraic integers by
Theorem \ref{main2} and  Proposition \ref{sp}. Hence $\tau_{M,L'}^{SU(2)}(\xi)$ is also an algebraic integer. 

Therefore throughout the remaining part of this section we will assume that
$r=\ord(\xi)$ is {\em even.} Note that in this case the order of $\xi^{1/4}$ is always $4r$ and
$e_8 \in \Z[\xi^{1/4}]$.

\begin{prop}\label{lem42}
Let $r=\ord(\xi)$ be even.
Suppose  $b=\pm p^s$, where $p$ is 0, 1 or a prime, $k$ an integer, and $\ve \in \{0,1\}$. Then
 $$\frac{H^{SU(2)}(k, b,\ve)}{H^{SU(2)}(0, \pm1,0)}\in
\Z[\xi^{1/4}]\  \quad {\text and}\quad \frac{H^{SU(2)}(k, 0,\ve)} {(1-\xi)\, \cD^{SU(2)}} \in
\Z[\xi^{1/4}]\ .$$
\end{prop}

The following lemma will be used in the proof of the above proposition for odd $b$.

\begin{lemma}
\label{divisor}
Suppose $b$ is odd, $r$ is even, $a\in \BZ$, and $f \in I_k$. Then
$$A:=\frac{\sumx^{\xi,SU(2)} q^{\frac{ b n^2}4 + \frac{an}{2}} \,  f(q^n,q)}{x_k O_\xi}$$ belongs to $\BZ_{(2)}[\xi^{1/4}]$, where $\BZ_{(2)}$ is the set of
all rational numbers with odd denominators.
\end{lemma}
\begin{proof}
Let $r= r_e r_o$, where $r_o$ is odd, and $r_e$ is a power of 2. Then $\ord(\xi^{1/4})=4r= (4r_e) r_o$, with $4r_e$ and $r_o$ coprime. By definition,
\begin{align}
\sumx^{\xi,SU(2)} q^{\frac{b n^2}4+ \frac{an}{2}} q^{nj}& = \frac14 \, G(b,4j+2a,\xi^{1/4}) \quad  \text{by Proposition \ref{lGauss} (d)}\notag\\
\quad 
&= \frac14  \, G(b \,r_o, 4j+2a, \xi^{r_o/4})\,  G(4b \,r_e, 4j+2a, \xi^{r_e})  \quad
\notag \\
 \quad
&= \frac14 \xi^{da^2/4} \, G(b \,r_o, 0, \xi^{r_o/4})\, \xi^{d(j^2+aj)}\, G(4b \,r_e, 4j+2a, \xi^{r_e})  \ ,
\label{e330}
\end{align}
where $d$ is any multiple of $r_o$ such that  $db \equiv -1 \pmod {4r_e}$.

Let us extend $\Delta(z)=z\otimes z$ to a  $\Zq$-algebra homomorphism
$$\Delta:\Zqz\to
\Zqz\otimes_{\Zq}\Zqz.$$
Define $Q(j) = dj^2 + da j$. Also define
$\Z[q^{\pm1/4}]$-module homomorphism
$T: \Z[z^{\pm 1}, q^{\pm 1/4}]\to\Z[q^{1/4}]$ by:
$$T(z^j)=G(4b \,r_e, 4j+2a, \xi^{r_e})\ .$$

Using \eqref{e330} we can rewrite
\be
 \sumx^{\xi,SU(2)} q^{\frac{b n^2}4+ \frac{an}{2}}
   f(q^n,q)   =\frac{\xi^{da^2/4}}{4}\, G(b \,r_o, 0, \xi^{r_o/4})\, \ev_\xi \Big\{  (\l Q \otimes T) (\Delta\,  f)\Big\}.
 \label{e234}
 \ee

It is enough to consider the case $f=z^m(q^l z;q)_k$.
Applying Corollary \ref{append} to $\Delta(z^m(q^l z;q)_k)$, 
using $z_1=z\otimes 1$ and $z_2=1\otimes z$,
we see that $A$  is
a $\Z[\xi^{\pm 1/4}]$-linear combination of terms of the form
\be
 B=\left(\frac{G(b \,r_o, 0, \xi^{r_o/4})}{4  O_\xi}\right)
\left(\frac{(\xi;\xi)_k\,  \ev_\xi \left \{\l Q (z^m(z;q)_{k_1}) \right\}}{ x_k (\xi;\xi)_{k_1}} \right)
\left(\frac{T(z^m(z;q)_{k_2})}{(\xi;\xi)_{k_2}}\right) \label{e8001}
\ee
with $k_i\le k$. There are three factors on the RHS of \eqref{e8001}
 and we will show that each factor belongs to $\BZ_{(2)}[\xi^{1/4}]$.
The last factor in $B$ is in $\BZ[\xi]$. In fact, if $z^m(z;q)_{k_2} = \sum_j c_j(q) z^j$ then
\bas
T(z^m(z;q)_{k_2}) &= \sum_j c_j(\xi)\, G(4 br_e,4j+2a,\xi^{r_e})
 = \sum_j c_j(\xi)\sum_{n=0}^{r_o-1}
\xi^{2 r_e(br_en^2+(2j+a)n)} \\
&= \sum_n \xi^{4 br_e^2n^2+2r_e a n}\sum_j c_j(\xi) \xi^{4 r_en  \, j } = 
\sum_n \xi^{4br_e^2n^2+2r_e  n (a+2m)} (\xi^{4r_en};\xi)_{k_2}\ ,
\end{align*}
which is divisible by $(\xi;\xi)_{k_2}$ in $\Z[\xi]$. This shows that the last factor of \eqref{e8001} is in $\BZ_{(2)}[\xi^{1/4}]$.

By Theorem \ref{thm1}, $\ev_\xi \big(\l Q  (z^m(z;q)_{k_1})\big)$ 
is in $x_{k_1}\, \BZ[\xi]$.
Hence the second factor  is in
$$\frac{(\xi;\xi)_k\,x_{k_1}}{x_k (\xi;\xi)_{k_1}} \BZ[\xi]=
 \frac{(\xi;\xi)_{\lfloor k/2\rfloor}}{(\xi;\xi)_{\lfloor k_1/2\rfloor}}\, \BZ[\xi] \subset \BZ[\xi].$$

By Proposition \ref{lGauss} (b),
$ G^2(b \,r_o, 0, \xi^{r_o/4})\sim 8 r_e$, while $O_\xi^2 \sim r/2$ by Lemma \ref{ab} (e).
 It follows that the square of the
first factor,  and hence the first factor itself, is in $\frac{r_e}{r} \BZ[\xi]= \frac{1}{r_o} \BZ[\xi] \subset \BZ_{(2)}[\xi].$
Here we use the fact that $\BZ_{(2)}[\xi]$ is integrally closed.

We can conclude that $B$, and hence $A$, is in $\BZ_{(2)}[\xi].$
\end{proof}

\begin{proof}[Proof of Proposition \ref{lem42}]
By Lemma \ref{h1}  there is $f_\ve(z,q)\in I_{2k+1+\ve}$ such that
\be
\frac{H^{SU(2)}(k, b,\ve)}{H^{SU(2)}(0, \pm1,0)}  \sim \frac{\frac14
\sum_{n=0}^{4r-1 } \xi^{\frac{b}{4} (n^2-1)} \xi^{-3 \ve n/2} f_\ve(\xi^n,\xi)}  {x_{2k+1+\ve}O_\xi} \label{e225}\ .
\ee

We split
the proof  into 3 cases:
 (1) $b\equiv 0\mod 4$;  (2) $b=\pm 2$ and (3) $b$ is odd.

\noindent
 (1) $b=4 b', b' \in \BZ$. Since  $H^{SU(2)}(k, b,1)=0$ by Lemma \ref{h1} (d), we can assume $\ve=0$. By \eqref{e225},
$$
\frac{H^{SU(2)}(k, b,0)}{H^{SU(2)}(0, \pm1,0)}   \sim \frac{\frac14
\sum_{n=0}^{4r-1} \xi^{b' (n^2-1)}  f_0(\xi^n,\xi)}  {x_{2k+1}O_\xi} \notag
 = \frac{
\sum_{n=0}^{r-1} \xi^{b'(n^2-1)}  f_0(\xi^n,\xi)}  {x_{2k+1}O_\xi}\ ,
$$
which is in
$\Zx$ by Proposition \ref{div1}.
\noindent
\vskip2mm

(2) $b=\pm 2$. Again  $H^{SU(2)}(k, b,1)=0$ by Lemma \ref{h1} (d), and we can assume $\ve=0$. This case was studied in \cite{bbl2}, where the exact value of $H^{SU(2)}(k, \pm 2,0)$ was obtained.
By Lemma 5.2 in \cite{bbl2}
we have
\begin{align*}
 H^{SU(2)}(k, \pm 2,0) & \sim  2 \sqrt {r}
 \prod^{k}_{i=0}\frac{1-\xi^{(2i+1)/2}}{1-\xi^{2i+1}}= 2 \sqrt {r}
 \prod^{k}_{i=0}\frac{1}{1+\xi^{(2i+1)/2}}\ .
 \end{align*}
 Hence from Lemma \ref{h1}, with $k \le r/2-1$,

\begin{align*} \frac{H^{SU(2)}(k, \pm 2,0)}{H^{SU(2)}(0, \pm1,0)}
\sim \frac{\sqrt r /O_\xi}{\prod^{k}_{i=0}( 1+\xi^{(2i+1)/2})} \in z\, \BZ[\xi^{1/4}] \end{align*}
where
$$ z=  \frac{\sqrt r /O_\xi}{\prod^{r/2-1}_{i=0}( 1+\xi^{(2i+1)/2})}.$$
The square of the numerator of $z$ is $r/O_\xi^2 \sim 2$, by Lemma \ref{ab}.

Let us calculate the square of the denominator.
For any integer $j$ one has $$ (1+ \xi^{(2j+1)/2}) \sim (1- \xi^{(2j+1)/2}).$$
Hence
\bas \left( \prod_{j=0}^{r/2-1} (1+ \xi^{(2j+1)/2})\right)^2 & \sim \prod_{j=0}^{r/2-1} (1+ \xi^{(2j+1)/2})(1- \xi^{(2j+1)/2}) = \prod_{j=0}^{r/2-1} (1- \xi^{(2j+1)})\\
& \sim  \frac{\prod^{r-1}_{j=1}(1-\xi^{j})}{ \prod^{r/2-1}_{j=1}(1-\xi^{2j})}= \frac{r}{r/2}=2.
 \end{align*}
We can conclude that $\frac{H^{SU(2)}(k, \pm 2,0)}{H^{SU(2)}(0, \pm1,0)} \in \BZ[\xi^{1/4}]$.
\vskip2mm

\noindent
{(3) Assume that $b$ is odd.}  
Splitting the sum in the numerator of the right hand side of \eqref{e225}  into even and odd $n$ we get
\begin{align}
\qquad & \frac14 \sum_{n=0}^{4r-1} \xi^{b(n^2-1)/4} \,\xi^{-3\ve n/2}
 f_\ve(\xi^n,\xi) \notag\\
 & = \frac{1}4 \left\{\xi^{-b/4} \sum_{n=0}^ {2r-1} \xi^{bn^2-3\ve n} \, f_\ve(\xi^{2n},\xi) + 
 \xi^{-3\ve/2}\sum_{n=0}^ {2r-1}
\, \xi^{b(n^2+ n)-3\ve n}\,  f_\ve(\xi^{2n+1},\xi) \right\} \notag \\
& =\frac{1}2 \left\{\xi^{-b/4} \sum_{n=0}^ {r-1} \xi^{bn^2-3\ve n} \, f_\ve(\xi^{2n},\xi) + \xi^{-3\ve/2}
\sum_{n=0}^ {r-1}
\, \xi^{b(n^2+ n)-3\ve n}\,  f_\ve(\xi^{2n+1},\xi) \right\}\label{e226}
\end{align}

Since $f_\ve(z^2,q)$ and $f_\ve(z^2q,q)$ belong to $I_{2k+1+\ve}$
 (according to Proposition \ref{ideal}), each summand in the  curly brackets
 of the right hand side of \eqref{e226} is divisible by
$x_{2k+1+\ve} O_\xi$, by Proposition \ref{div1}. It follows from \eqref{e225} that
$$
 \frac{H^{SU(2)}(k, b,\ve)}{H^{SU(2)}(0, \pm1,0)}\in \frac{1}{2}\;
\Z[\xi^{1/4}],
$$
which, together with
 Lemma \ref{divisor}, implies

$$ \frac{H^{SU(2)}(k, b,\ve)}{H^{SU(2)}(0, \pm1,0)} \in  \frac{1}{2}\;
\Z[\xi^{1/4}] \cap \Z_{(2)}[\xi^{1/4}] = \Z[\xi^{1/4}].$$

Finally $$\frac{H^{SU(2)}(k, 0,\ve)} {(1-\xi)\, \cD^{SU(2)}}\in \Zxf$$
 follows from Lemma \ref{h1} (c), which says that  $ {H^{SU(2)}(0, \pm1,0)}\sim (1-\xi) \, \cD^{SU(2)}$.
\end{proof}


\subsection*{Proof of Theorem  \ref{main1} }
By Proposition \ref{lem42}, each factor in the right hand side  of \eqref{e_diag} is in $\Zxf$, hence
$\tau^{SU(2)}_{M, L'}(\xi)\in \Zxf$ if $M$ is diagonal of prime type.

Now suppose $M$ is an arbitrary 3-manifold. According to Proposition \ref{p402}, there exist lens spaces
$\L(2^{k_1},-1), \ldots, \L(2^{k_j},-1)$, such that
$M^{\# 2s} \# N$
is diagonal of prime type for every positive integer $s$. Here $$N:=\#_{i=1}^j \L(2^{k_i},-1)\ .$$
 By Lemma \ref{lens}, there is an odd colored knot $K_i \subset \L(2^{k_i},-1)$ such that
$\tau^{SU(2)}_{\L(2^{k_i},-1), K_i} \neq 0$.
The knots $K_i$ together form a link $L'' \subset N$, and
$$\tau^{SU(2)}_{N,L''}(\xi) = \prod_i \tau^{SU(2)}_{\L(2^{k_i},-1), K_i} \neq 0.$$

Taking the connected  sum of $(N,L'')$ with $2s$ copies of $(M,L')$, we get a diagonal
3-manifold of prime type. Hence,
 $$ \left (\tau_{M,L'}^{SU(2)}(\xi)\right)^{2s}
 \, \tau_{N,L''}^{SU(2)}(\xi)  \in \Z[\xi^{1/4}]\ $$
for every positive integer $s$. Applying Lemma \ref{ab} (c),
 we get
$\tau^{SU(2)}_{M, L'}(\xi) \in  \Zxf$.

\qed


\appendix

\numberwithin{equation}{section}
\numberwithin{figure}{section}

\section{Proof of
Theorem \ref{prop.habiro}}


\def\inv{\mathrm{inv}}
\def\trq{\tr_q}

\subsection{Algebraic preliminaries}
We first recall the universal quantized algebra $U_h= U_h(sl_2)$ and some of its properties. For more details see e.g. \cite{ha}.

The universal quantized algebra  $U_h = U_h (sl_2 )$ is the
$h$-adically complete $\Q[[h]]$-algebra, topologically generated by the elements $H, E,$
and $F$, satisfying the relations
$$ HE-EH= 2E, \quad HF-FH= -2F, \quad EF-FE = \frac{K- K^{-1}}{v - v^{-1}}\ ,$$
where $K:= \exp(hH/2)$,  $v:= \exp(h/2)$, and $v^2=q$.
The algebra $U_h$ has a structure of Hopf algebra, which makes $U_h$ into a $U_h$-module via the adjoint representation, and
defines a tensor product on the set of  $U_h$-modules. In particular, the completed tensor powers $U_h ^{\hat \otimes m}$ 
is a $U_h$-module via the adjoint representation.
For a set $Y \subset U_h ^{\otimes m}$ its subset of invariant elements is defined by
$$ Y^\inv := \{ y \in Y \mid a \cdot y = \epsilon(y), \quad \forall a \in U_h\}\ ,$$
where $\epsilon$ is the antipode of $U_h$ and $a \cdot y$ is the adjoint action.
It is known that
$(U_h)^\inv$ is exactly the center of $U_h$.

For each positive integer $n$ there is a unique $n$-dimensional irreducible $U_h$-module, denoted by $V_n$, we set
$V: = V_2$. Let
$$R={\operatorname {Span}}_{\BZ[v^{\pm 1}]} \{V_n, n\geq 1\}\ ,$$
which is a $\BZ[v^{\pm 1}]$-algebra whose
multiplication is the  tensor product. One has
\be V_n V = V_{n+1} + V_{n-1}\ ,
\label{e811}
\ee
and as a ring $R = \BZ[v^{\pm 1}][V]$, the ${\BZ[v^{\pm 1}]}$-polynomial algebra in $V$.

 For an $U_h$-module $W$ and $x \in U_h$ the quantum trace is defined by
$$ \trq^W(x) = \tr(xK^{-1}, W)\ ,$$
which can be linearly extended to the case when $W$ is a ${\BZ[v^{\pm 1}]}$-linear combination of $U_h$-modules.

The quantum trace preserves ad-invariance, which means the following.
Suppose $W \in R$ and $y \in \big( U_h ^{\otimes m}\big)^\inv$, then $\big( \id^{\otimes (m-1)}\otimes \trq^W\big) (y) \in \big( U_h ^{\otimes (m-1)}\big)^\inv$. 

\subsection{New bases for $R$}

In $R$ consider the following elements: $ P_0^{(0)} = P_0^{(1)}=1$,
$$ P_n^{(0)} = \prod_{j=1}^n (V - \lambda_{2j-1})\ , \quad  P_n^{(1)} = \prod_{j=1}^n (V - \lambda_{2j})\ ,$$
where $\lambda_n=v^n+v^{-n}$.
Note that $P_n^{(0)} $ is  $P_n$ of \cite{ha}.
Since $P_n^{(0)}$ is a monic polynomial of degree $n$ in $V$ with coefficients in $\BZ[v^{\pm 1}]$, 
it is clear that the set $\{ P_n^{(0)}, n =0,1,2,\dots \}$ forms a $\BZ[v^{\pm 1}]$-basis of $R$.
Similarly, $\{ P_n^{(1)}, n =0,1,2,\dots \}$ also forms a $\BZ[v^{\pm 1}]$-basis of $R$. 
It is not difficult to express $V_{n}$ through these bases. In fact, \eqref{e811}, together with an easy induction, 
will give the following identities, the first of which was  obtained in \cite{ha}.

\be\label{app_even}
V_n=\sum^{n-1}_{k=0} \qbinom{n+k}{2k+1}  P^{(0)}_k\ ,
\quad V_n=\sum^{n-1}_{k=0} \qbinom{n+k}{2k+1}\frac{\lam_n}{\lam_{k+1}} P^{(1)}_k \ . 
\ee

\subsection{Integral subalgebras and their completions}
Following \cite{ha} let $\Uq0$ be the $\Z[q^{\pm 1}]$-subalgebra of $U_h$ generated by
$\tF l$, $e$ and $K^{\pm 2}$, where
$$
\tF l := q^{l(1-l)/4} F^lK^l/[l]!\quad \text{and}\quad e := \cb1 E\ .
$$
Let $\Uq1 = K\, \Uq0$ and
 $\cU_q = \Uq{0}\oplus \Uq{1}$.

\def\bve{\mathbf{e}}
\def\tUe{\widetilde{\U_q^{\otimes (\bve)}}}

Let $\F_p(\U_q^{\otimes m}) \subset \U_q^{\otimes m} $ be the $\Z[q^{\pm 1}]$-span of 
elements of the form $y_1 \otimes y_2 \otimes \dots \otimes y_m$, where each $y_j$ belongs to $\U_q$, and one of them belongs to $\U_q \, e^p \U_q$. 
For a set $Y \subset \U_q^{\otimes m}$ define its completion
$$\tilde Y := \left \{ \sum_{j=0}^\infty z_p\mid z_p \in Y \cap \F_p(\U_q^{\otimes m})    \right\}.$$

In particular, when $m=1$, on can define  $\tilde\U_q$ and $\tilde{\U}^{(\ve)}_q$
for $\ve\in\{0,1\}$. For $\bve=(\ve_1,\dots,\ve_m) \in \{0,1\}^m$,
 let $\tUe$ be the completion of 
$\Uq{\ve_1 }\otimes \dots \otimes \Uq{\ve_m}$ defined as above.

The center $Z(\U_q)= (\U_q)^\inv$ is freely generated as an $\Zq$-algebra
by  the quantum Casimir operator
$$C=(1-q^{-1})\tilde F^{(1)}K^{-1}e+K+q^{-1}K^{-1}\, \in \Uq{1}\ .$$


Set
$$\sigma^{(0)}_n=\prod^n_{i=0}(q C^2-(q^i+2+ q^{-i}))  \quad \text{and} \quad
\sigma^{(1)}_n=C\sigma^{(0)}_n\ . $$

 Theorem 1.1 in \cite{hab1} states that

 \be (\tilde\U^{(\ve)}_q)^\inv  = \left \{
\sum_{p\geq 0} a_p\,  \sigma_p^{(\ve)} \mid a_p \in \Zq
\right \}.
\ee

We will need the following result.

\begin{prop}\label{8.5}
Suppose $x \in \Uq{\ve}$, $\ve \in \{0,1\}$. Then for every $n$,
$\tr^{P_n^{(\ve)}}_q(x)$ belongs to $(q;q)_n\,  \Zq$.
\end{prop}

\begin{proof} 
If $\ve=0$, this is  \cite[Lemma 8.5]{ha}.
The case $\ve=1$ can be proved similarly. It
is enough to set $x=\tF l K^{2j+1} e^{l'}$.
It is easy to see that $\tr^{\Pe_n}_q(\tF l K^{2j+1} e^{l'})=0$ if $l\ne l'$.
Set
$$
B(n,l,j) := \tr^{\Pe_n}_q(\tF l K^{2j+1} e^l)\ .
$$
Then it is clear that $B(n,l,j)=0$ when $l>n$ because $e^l$ vanishes on $V_1, V_2, \dots, V_{n+1}$.
When $l\le n$, by a similar argument as in the proof of \cite[lemma 8.8]{ha}, we have
$$
B(n,l,j) = \cb{j-n} \cb{j+n} B(n-1, l, j) + q^j (1-q^{-l}) B(n-1, l-1, j+1)\ .
$$
The above recursive relation and a simple induction will show that
$$
B(n,l,j)
= q^{-(j + l) n} (q;q)_n (q; q)_{n-l} \binomq{j-1}{n-l} \binomq{j+n}{n-l} \in (q;q)_n\,  \Zq\ .
$$
\end{proof}

\begin{lemma} For every non-negative integers $k$, $p$ and $\ve \in \{0,1\}$, one has
\be\label{2221}
 \tr_q^{P^{(\ve)}_{k}}(\sigma^{(\ve)}_p) = \delta_{k,p} \frac{\{2k+1\}!}{v^\ve \{1\}} \, \lambda_{k+1}^\ve\ .\ee
\label{rosso}
\end{lemma}
\begin{proof}
The case $\ve =0$ is proved in \cite[Proposition 6.3]{ha}.
Hence, we restrict to $\ve=1$. 

As explained in \cite[Section 6.3.1]{ha}, there exists 
a homomorphism $\varphi: R\to Z(\U_q\otimes\Z[v^{\pm1}])$
sending $V$ to $v C$. In particular, for
$$ S_n^{(1)} : = V \prod_{j=1}^n (V^2 - (\lambda_j)^2)$$
we have $\varphi(S^{(1)}_n)=
v \sigma^{(1)}_n$. Moreover,
for any $x,y \in R$,
$$\tr_q^x(\varphi(y))=J_\cH(x,y):=\la x,y\ra\ ,$$
where $\la x,y\ra$ is 
the Rosso pairing defined as
  the colored Jones polynomial of the 
0-framed Hopf link $\cH$, whose two components are colored by $x$ and $y$.
Note that this pairing is symmetric.
Hence, for $\ve=1$, the left hand side of \eqref{2221} is
equal to $\la  P^{(1)}_k, v^{-1} S_p^{(1)}\ra $.

Since $\lambda_n := v^{n} + v^{-n}= \la V_n, V\ra /[n]$,
 for every $f(V)\in R$ we have
\be\label{in}
\la V_n, f(V)\ra = [n] f(\lambda_n)\ .
\ee
Hence if $m < n$, then $\la V_{2m+2}, P_n^{(1)} \ra =0$
and $\la  S_n^{(1)}, V_{m+1} \ra =0$.

Using $V_n V = V_{n+1} + V_{n-1}$, 
we get
\be\label{Pe}
P_n^{(1)} = V_{n+1} + \text{a $\Z[v^{\pm 1}]$-linear combination of } V_1,V_2\dots, V_n\ ,
\ee
and
$$S_n^{(1)}= V_{2n+2} + \text{a $\Z[v^{\pm 1}]$-linear combination of } V_2,V_4, \dots, V_{2n}\  .$$

Therefore $\la S_m^{(1)}, P_n^{(1)}\ra = 0$ if $m\ne n$.

Finally
\begin{align*}
 v \tr_q^{P^{(1)}_{n}}(\sigma^{(1)}_n) &= \la S_n^{(1)}, P_n^{(1)}\ra = \la S_n^{(1)}, V_{n+1}\ra = [n+1] \lambda_{n+1} \prod_{j=1}^n ((\lambda_{n+1})^2 - (\lambda_j)^2)\\ &= [n+1] \lambda_{n+1}\prod^n_{j=1} \{j\}\{2n+2-j\}
\end{align*}

\end{proof}

\subsection{Proof of Theorem \ref{prop.habiro}} Suppose $L\sqcup L'$ is an oriented framed link with fixed colors $\bs=(s_1,\dots,s_l)$ on $L'$. Here $L$ has $m$ ordered components and $\ve_i$ are defined in \eqref{3306}.

According to \cite[Theorem 4.1]{ha}, there is an element $J_T \in \big( U_h ^{\otimes (m+l)}\big)^\inv$ such that
$$
J_L(\bn) = \trq^{V_{n_1} \otimes \cdots \otimes V_{n_m}\otimes V_{s_1}
\otimes \cdots \otimes V_{s_l}} (J_T)\ .
$$
(In \cite{ha}, $J_T$ is the universal invariant of a bottom tangle whose closure is $L \sqcup L'$.)

Using \eqref{app_even} to express $V_{n_i}$ as a linear combination of $P_{k}^{(\ve_i)}$, we have
\begin{align}
J_L(\bn) & =  \sum_{k_i=0}^{n_i-1} \trq^{P_{k_1}^{(\ve_1)} \otimes \cdots \otimes  P_{k_m}^{(\ve_m)}
\otimes V_{s_1}
\otimes \cdots \otimes V_{s_l}} (J_T)\,  \prod_{i=1}^m \qbinom{n_i+k_i}{2k_i+1}
\;\;\frac{\lam^{\ve_i}_{n_i}}{\lam^{\ve_i}_{k_i+1}} \notag \\
& = \sum_{k_i=0}^{n_i-1} \trq^{P_{k_1}^{(\ve_1)'} \otimes  \cdots \otimes P_{k_m}^{(\ve_m)'}
\otimes V_{s_1}
\otimes \cdots \otimes V_{s_l}} (J_T) \, \prod_{i=1}^m \qbinom{n_i+k_i}{2k_i+1}
\;\cb{k_i}!\;\frac{\lam^{\ve_i}_{n_i}}{\lam^{\ve_i}_{k_i+1}}\ ,
\label{e5510}
\end{align}
which is \eqref{eq.jj} with
$$c_{L\sqcup L'} (\bk) = \trq^{P_{k_1}^{(\ve_1)'} \otimes  \cdots \otimes P_{k_m}^{(\ve_m)'}
\otimes V_{s_1}
\otimes \cdots \otimes V_{s_l}} (J_T)\ .$$
Here $P_{k}^{(\ve)'}=P_k^{(\ve)}/\{k\}!$ .

Without loss of generality, we may assume
$k_1=k=\max(k_1,\ldots,k_m)$.

By Theorem A.3 in \cite{bbl1},  $(\id^{\otimes m}\otimes \tr^{V_{s_1}}_q
\otimes \cdots \otimes \tr^{V_{s_l}}_q)(J_T) \in q^a \, \big( \tUe\big)^\inv$, for some  $a \in \frac 14\Z$. Let

$$ y:= (\id \otimes \tr^{P^{(\ve_2)'}_{k_2}}_q\otimes \cdots \otimes \tr^{P^{(\ve_m)'}_{k_m}}_q
\otimes \tr^{V_{s_1}}_q
\otimes \cdots \otimes \tr^{V_{s_l}}_q)(J_T)\ .$$
Then
$$c_{L\sqcup L'}(\bk)
= \tr^{P^{(\ve_1)'}_{k_1}}_q (y)\ .$$

Proposition \ref{8.5}, as well as the fact that quantum trace preserves ad-invariance, gives us $y \in q^a\, \big( \U_q^{(\ve_1)}\big)^\inv$. Hence $y$ has a presentation
 $y= q^a \, \sum_{p\geq 0} d_p \sigma_{p}^{(\ve_1)}$ with $d_p\in \Z[q^{\pm 1}]$. We then have
$$\tr^{P^{(\ve_1)'}_{k_1}}_q (y)= q^a \sum_p d_p \, \tr_q^{P^{(\ve_1)'}_{k_1}}\big( \sigma_p^{(\ve_1)}\big),$$
which belongs to $\frac{(q^{k+1};q)_{k+1}}{1-q}\, \Zqf$ by Lemma \ref{rosso}.\qed



\end{document}